\documentclass{amsart}
\usepackage[usenames, dvipsnames]{color}

\usepackage{hyperref}

\usepackage[numbers, sort&compress]{natbib}

\title{What is a lattice W-algebra?}
\author{ Anton Izosimov and Gloria Mar\'i Beffa}
\address{Department of Mathematics\\ University of Arizona\\ Tucson AZ 85716}\address{Mathematics Department\\ University of Wisconsin\\ Madison WI 53706}
\thanks{The authors gratefully acknowledge support:\, AI through National Science Foundation grant DMS-2008021, GMB  through research funding from the College of Letters \& Science at UW-Madison}
\usepackage{amsmath,amsfonts,amsthm,amssymb,relsize,mathdots,MnSymbol}

\newtheorem{theorem}{Theorem}
\numberwithin{theorem}{section}
\newtheorem{corollary}[theorem]{Corollary}

\newtheorem{lemma}[theorem]{Lemma}

\newtheorem{comment}[theorem]{Comment}
\newtheorem{proposition}[theorem]{Proposition}

\theoremstyle{definition}
\newtheorem{definition}[theorem]{Definition}
\newtheorem{example}[theorem]{Example}
\newtheorem{remark}[theorem]{Remark}

\def\RR{\mathbb R}
\def\RP{\mathbb{RP}}
\def\Z{\mathbb Z}

\def\SL{\mathrm {SL}}
\def\GL{\mathrm {GL}}
\def\PSL{\mathrm {PSL}}
\def\r{\widehat{r}}

\def\tr{\mathrm{Tr}}

\def\T{\mathcal{T}}
\def\L{\mathcal{L}}
\def\So{\mathcal{S}}

\def\F{\mathcal{F}}
\def\G{\mathcal{G}}
\def\N{\mathcal{N}}
\def\H{\mathcal{H}}

\def\sl{\mathfrak{sl}} 
\def\b{\mathfrak{b}}

\def\g{\mathfrak{g}}
\def\h{\mathfrak{h}}

\def\a{\mathbf{a}}
\def\0{\mathbf{0}}

\def\hq{\hat q}

\renewcommand{\P}{\mathbb{P}}

\newcommand{\PGL}{\mathrm{PGL}}
\newcommand{\Id}{\mathrm{Id}}
\newcommand{\gl}{\mathrm{gl}}

\newcommand{\Tr}[1]{\mathrm{Tr}\,#1}
\newcommand{\DO}[2]{\mathrm{DO}(#1, #2)}
\newcommand{\RDO}[2]{\mathrm{RDO}(#1, #2)}
\newcommand{\IDO}[2]{\mathrm{IDO}(#1, #2)}
\newcommand{\PBDO}[2]{\mathrm{PBDO}(#1, #2)}
\newcommand{\PSIDO}[1]{\Psi\mathrm{DO}({#1})}

\newcommand{\PDO}[1]{\mathrm{DO}({#1})}
\newcommand{\IPDO}[1]{\mathrm{IDO}({#1})}
\newcommand{\IPSIDO}[1]{\mathrm{I}\Psi\mathrm{DO}({#1})}
\renewcommand{\log}{\ln}

\begin{document}
\begin{abstract} 
We employ the Poisson-Lie group of pseudo-difference operators to define lattice analogs of classical $W_m$-algebras. We then show that the so-constructed algebras coincide with the ones given by discrete Drinfeld-Sokolov type reduction.
\end{abstract}
\maketitle
%\ai{We need to discuss the content of the paper and then decide on the title. I think the main object of interest for readers here would be $W_n$-algebras, which corresponds to the ${SL}$ case. So I am not sure we should only discuss the $GL$ case here. Another thing we need to decide is whether we want to include a discussion on integrable systems and the companion bracket, or just the quadratic one (and hence no integrable systems). I personally prefer shorter and more focused papers, but I would be okay with both options. As for the title, I think ``What is a lattice W algebra?'' sounds nice, but it may not reflect the content if we include more stuff on integrable systems etc.}
%\textcolor{magenta}{From Gloria: I am happy with your title. Perhaps we can write the portion on GL-SL first, and then see how much we have. If we do not have enough we include the companion, if we do we close the paper and start another one on integrability. As of the GL-SL part, in my mind the structure is as follows: 1) you define your bracket on scalar operators for GL, and I define it on our side also (we did SL directly, so we need to show it works for GL). 2) We prove they are the same. 3) We then perform the two reductions, one on each side, and we show it is the same reduction and hence the SL are also the same bracket. Is that what you had in mind?}
%\ai{Yes, more or less.}

\tableofcontents

\section{Introduction and Background}
It is well-known that the space of monic order $m$ periodic differential operators without sub-leading term, i.e. operators of the form $$u_0(x) + \ldots + u_{m-2}(x) \partial^{m-2} + \partial^m,
$$
where $\partial := \partial / \partial x$ and the coefficients $u_i(x)$  are periodic functions, has a remarkable Poisson structure, called the second Adler-Gelfand-Dickey bracket \cite{GD, Adler}. The corresponding Poisson algebra is known as the classical $W_m$-algebra and can be identified with the semi-classical limit of the $W_m$-algebra arising in conformal field theory \cite{zamolodchikov1985infinite}. In connection with integrable systems, Adler-Gelfand-Dickey structures are best known as Poisson brackets for KdV-type equations. In particular, the $W_2$-structure coincides with the second Poisson bracket for the classical KdV equation and can be interpreted as the Lie-Poisson structure on the dual of the Virasoro algebra. The present paper concerns the problem of discretization of Adler-Gelfand-Dikii brackets. 

\par

Recall that in the continuous setting the Adler-Gelfand-Dickey bracket can be constructed in two different ways: either from the multiplicative structure on the algebra of formal pseudo-difference operators (which can also be interpreted as a Poisson-Lie structure on the extended group of such operators \cite{khesin1995poisson}), or by performing the so-called Drinfeld-Sokolov reduction to a specific symplectic leaf in the dual of an affine (Kac-Moody) Lie algebra \cite{DS}\footnote{There exists yet another equivalent definition of a $W$-algebra as the center of a certain completion of the universal
enveloping algebra of an affine algebra \cite{feigin1992affine}. We do not discuss this construction or its possible discretizations in the present paper. }. The goal of the present paper is to show that, when properly understood, both constructions admit a discretization (a lattice version). Moreover, we show that, similarly to the continuous setting,  both discrete constructions yield the same result. Thus, there is a well defined notion of a \textit{lattice $W_m$-algebra}. As special cases, one recovers familiar structures. Namely, for $m =2$ one obtains the lattice Virasoro algebra of Faddeev-Takhtajan-Volkov~\cite{volkov1988miura, faddeev2016liouville} (also known as a cubic Poisson structure associated to the Volterra lattice \cite{suris1999integrable}), while $m =3$ corresponds to the lattice $W_3$-algebra of Belov-Chaltikian~\cite{belov1993lattice}.

The discrete version of the Drinfeld-Sokolov reduction is not new: in a 2013 paper with Wang \cite{beffa2013hamiltonian}, followed by a 2020 paper with Calini \cite{beffa2020compatible}, the second author showed
the existence of two compatible Poisson structures for lattice KdV equations, one
of which can be thought of as a discretization of the second Adler-Gelfand-Dickey
bracket and is obtained via a Drinfeld-Sokolov type reduction process. The new ingredients of the present paper is the second construction based on the Poisson-Lie group of scalar operators, as well as the proof of equivalence of these two approaches. Below we describe these two constructions both in the continuous and discrete settings and formulate our main theorem on the coincidence of two definitions of lattice $W_m$-algebras. %A continuous version of that result is due to \ai{I will add a reference later. We may also need to talk about what's known about the $q$-case, here or later.}\par

First, we recall the construction of continuous $W_m$-algebras (Adler-Gelfand-Dickey brackets) based on scalar differential operators. Consider the associative algebra of periodic formal pseudo-differential symbols, i.e. infinite in the negative direction Laurent series over periodic functions in terms of the differentiation operator $\partial := \partial/\partial x$. This algebra has a natural decomposition into a direct sum of two subalgebras: the subalgebra of differential operators, and the complementary subalgebra of integral operators. The difference of projectors onto these subalgebras defines an $r$-matrix which is skew-symmetric with respect to Adler's trace form. Therefore, the associated Sklyanin bracket gives a multiplicative Poisson structure on periodic pseudo-differential symbols, and in particular on periodic differential operators. Furthermore, that bracket restricts to monic operators. The final step of the construction is to take the quotient by the adjoint action
$$
u_0 + \ldots + u_{m-1} \partial^{m-1} + \partial^m \mapsto f \circ (u_0 + \ldots + u_{m-1} \partial^{m-1} + \partial^m) \circ f^{-1},
$$
where $f$ is a non-vanishing function with periodic logarithmic derivative. The corresponding quotient space can be identified with the space of operators without sub-leading term, and the quotient Poisson structure arising on that space is precisely the Adler-Gelfand-Dickey bracket. Alternatively, one can take the quotient by the same action but with periodic $f$. This leads to the space of operators with constant sub-leading term, in which operators with zero sub-leading term are embedded as a Poisson submanifold \cite{khesin1995poisson}.\par
The space of monic order $m$ periodic differential operators without sub-leading term has a geometric interpretation. Namely, for any such operator let $f_1, \dots, f_m$ be a basis in its kernel. Then $\mu(x) := (f_1(x) : \ldots :f_m(x))$ is a curve in the projective space $\RP^{m-1}$. Upon a change of basis, this curve undergoes a projective transformation. Furthermore, due to the periodicity of the differential operator, this curve is quasi-periodic i.e. closed up to a projective transformation, known as the \textit{monodromy}. One can show that this construction defines a one-to-one correspondence between monic order $m$ periodic differential operators without sub-leading term and projective equivalence classes of quasi-periodic curves in $\RP^{m-1}$, satisfying a certain non-degeneracy condition~\cite{ovsienko1990symplectic}.  Thus, the Adler-Gelfand-Dickey bracket can be viewed as a Poisson structure on equivalence classes of curves in the projective space.

There has been a large body of work on discrete analogs of  the Adler-Gelfand-Dickey bracket, i.e. discrete $W_m$-algebras, both in the $q$-difference \cite{PST, semenov1998drinfeld, frenkel1996quantum, frenkel1998drinfeld} and difference  \cite{faddeev2016liouville, belov1993lattice, beffa2013hamiltonian, volkov1988miura, hikami1997classical, mansfield2013discrete} settings. However, it seems that the above direct construction producing the Adler-Gelfand-Dickey bracket from the multiplicative Poisson structure on differential operators has never been discretized. The main difficulty encountered on this path is that although both difference and $q$-difference operators admit natural multiplicative Poisson structures, these structures do not restrict to monic operators. This difficulty is usually overcome by considering various non-multiplicative modifications of the Sklyanin bracket, cf. e.g. \cite{PST}. The goal of the first part of the present paper (Section \ref{sec:scalar}) is to show that this problem can in fact be resolved fully within the multiplicative setting. To that end, we slightly change the perspective on the classical Adler-Gelfand-Dickey bracket. Namely, observe that instead of reducing \textit{monic} operators with respect to the adjoint action, one can consider \textit{all} order $m$ operators with non-vanishing leading term and reduce with respect to the \textit{left-right action}
 $$
u_0 + \ldots +  u_m\partial^m \mapsto f \circ (u_0 + \ldots +  u_m \partial^m) \circ g^{-1},
$$
where $f$ and $g$ are functions with periodic logarithmic derivatives and such that $f/g$ is periodic. This quotient can again be identified with monic operators without sub-leading term, and the reduced bracket coincides with the Adler-Gelfand-Dickey bracket. Indeed, the left action of functions on order $m$ operators admits a Poisson section given by monic operators, so reduction by that action is equivalent to restriction. Therefore, taking the quotient by the left-right action is equivalent to first restricting to monic operators and then taking the quotient by the adjoint action.  %An advantage of this different viewpoint on the Adler-Gelfand-Dickey bracket is that 

This modified construction of the Adler-Gelfand-Dickey bracket admits an immediate discretization. In the difference case, consider the space of upper-triangular $N$-periodic difference operators of degree $m$, i.e. operators of the form
$$
u^0 + \ldots +  u^m\T^m,
$$
where each $u^j$ is an $N$-periodic bi-infinite sequence of real numbers $u^j_i \in \RR$, and $\T$ is the left shift operator on bi-infinite sequences, i.e. $(\T v)_i := v_{i+1}$. This space has a Poisson structure inherited from the natural multiplicative bracket on all periodic difference operators (as in the differential case, the multiplicative bracket on difference operators can be viewed as a Poisson-Lie structure; however, just like differential operators, difference operators do not form a group and hence must be extended to pseudo-difference ones in order to obtain a Poisson-Lie interpretation). The latter multiplicative bracket is in fact well familiar: the algebra of all $N$-periodic difference operators is isomorphic to the  loop algebra of Laurent polynomials in one variable with values in $N \times N$ matrices, and the corresponding bracket on the loop algebra is given by the standard trigonometric $r$-matrix. In the setting of difference operators, this $r$-matrix and the corresponding Poisson structure were considered e.g. in \cite{oevel1997poisson, gekhtman2019periodic}.
\begin{definition}
The \textit{lattice $W_m$-algebra} is the quotient of degree $m$ upper-triangular $N$-periodic difference operators by the left-right action
\begin{equation}\label{eq:lraction}
u^0 + \ldots +  u^m\T^m \mapsto f \circ (u^0 + \ldots +  u^m\T^m) \circ g^{-1},
\end{equation}
where $f$ and $g$ are bi-infinite sequences satisfying the quasi-periodicity conditions $f_{i+N} /f_i= g_{i+N}/g_i = \lambda $ for some constant $\lambda \in \RR \setminus 0$.
\end{definition}
For $m = 2,3$, we recover familiar structures. Namely, for $m =2$ we obtain the lattice Virasoro algebra of Faddeev-Takhtajan-Volkov \cite{volkov1988miura, faddeev2016liouville} (see Section \ref{app:voltproof}). Likewise, for $m =3$ one gets the lattice $W_3$-algebra of Belov-Chaltikian \cite{belov1993lattice} (this is a long but straightforward calculation which we omit). We also believe that our algebras can be constructed via the $q \to \sqrt[N]{1}$ limit from $q$-difference $W_m$-algebras of~\cite{PST, semenov1998drinfeld, frenkel1996quantum, frenkel1998drinfeld}. Moreover, we expect that those $q$-difference $W_m$-algebras themselves can be obtained using a $q$-difference version of the above construction, i.e. by double reduction of the space of fixed order $q$-difference operators.
%Also, don't you get these structures after you have reduced further to $u_0$ constant?}

%\aiq{For the first question: not sure, I think this will complicate the exposition. For the second one: I am not sure I understand the question. When the degree and period are coprime, this quotient can be modeled as the space of monic operators with $u_0 = u_m = 1$. We can say this here.} \textcolor{magenta}{From Gloria: Yes, that is exactly what I meant.}

Also note that when the degree $m$ and the period $N$ are coprime, the action \eqref{eq:lraction} admits a section given by operators of the form
\begin{equation}\label{redmonic}
(-1)^m + v^1 \T + \ldots +  v^{m-1}\T^{m-1} + \T^m,
\end{equation}
% \textcolor{red}{From Gloria: It you want this to match our part I would use $(-1)^{m}+v_1 \T + \ldots +  v_{m-1}\T^{m-1} +\T^m$  instead. Of course it is the same, but this will give a direct connection to our projective polygons with $v_i = - a_i$.}
so one can in principle identify the lattice $W_m$-algebra with operators of such form. However, the Poisson bracket written in terms of entries of $v^1, \dots, v^{m-1}$ is non-local. As we will see below, a better coordinate system is provided by identifying the lattice $W_m$-algebra with the space of \textit{twisted polygons}, i.e. bi-infinite sequences of points $\mu_i \in \RP^{n-1}$ satisfying $\mu_{i+N} = \phi(\mu_i)$ for a certain projective transformation $\phi$ (the monodromy). The correspondence between difference operators and polygons is constructed in exactly the same way as between differential operators and curves:  given an order $m$ difference operator, one constructs a polygon by the rule $\mu_i := (f^1_i : \ldots :f^m_i)$, where $f^1, \dots, f^m$ is any basis of the kernel. \par
\smallskip 
Using a different point of view, the authors of \cite{beffa2013hamiltonian} constructed a pair of discrete Hamiltonian structures using a modified Drinfeld--Sokolov reduction process. In the continuous case, the reduction coincided with the Adler-Gelfand-Dickey bracket, as shown in \cite{MB2}. In the case of $G= \SL(m)$, the original Drinfeld--Sokolov reduction~\cite{DS} can be described as follows:  consider matrix differential operators of the form
\[
\L = \frac{d}{dx} + q(x) + \Lambda
\]
where the function $q\in C^\infty(S^1, \b)$ has values in the Borel subalgebra $\b \subset \sl(m)$ of upper triangular matrices, and $\Lambda = \sum_{i=1}^{m-1} E_{i+1,i}$, where by $E_{i,j}$ we mean a matrix with $1$ in the $(i,j)$ entry and zeroes elsewhere. While \cite{DS} includes a spectral parameter in their description, the spectral parameter will not be relevant in this paper and we will assume it to vanish. The authors of \cite{DS} showed  that any element in $C^\infty(S^1, \SL(m))$  can be transformed to a certain normal form using the gauge action
\[
\mathcal L \mapsto \So^{-1} \L \So
\]
where $\So \in C^\infty(S^1, \N)$, with $\N$ being the subgroup of $\SL(m)$ given by the exponential of strictly upper triangular matrices.  The normal form reads
\begin{equation}\label{can}
\L^{can} = \frac{d}{dx} + \sum_{i=1}^{m-1} v_i(x) E_{i,m} + \Lambda.
\end{equation}
Therefore, $\L^{can}$ can be seen as defining a  section transverse to the quotient $C^\infty(S^1, \b)/ C^\infty(S^1, \N)$, with $C^\infty(S^1, \N)$ acting on $C^\infty(S^1, \b)$ using the gauge action.

If $\g$ is the Lie algebra of $G$ with $G$ semisimple, the space $C^\infty(S^1, \g^\ast)$ has a natural Lie-Poisson structure
given by 
\[
\{\F, \G\}(\xi) = \int_{S^1}\langle [\delta \F, \delta \G], \xi\rangle dx.
\]
The cocycle $\omega:C^\infty(S^1,\g)\times C^\infty(S^1,\g) \to \RR$ given by
\[\omega(\xi, \mu) = \int_{S^1} \langle\frac{d\xi}{dx}, \mu\rangle dx
\]
defines the Poisson bracket
\begin{equation}\label{central}
\{\F, \G\}(\xi,s) = \int_{S^1}\left(\langle [\delta \F, \delta \G], \xi\rangle + s \langle\frac{d\delta \F}{dx}, \delta \G \rangle\right)dx.
\end{equation}
Here $\g$ and $\g^\ast$ are identified using an invariant inner product in $\g$ - in our case it is defined by the trace of the product. The bracket (\ref{central} is the Lie-Poisson bracket of the central extension  $C^\infty(S^1,\g^\ast)\times \mathbb{R}$, and it foliates into Poisson submanifolds corresponding to a fixed value of the real component $s$. When $s=1$ the symplectic leaves are given by gauge orbits, with the gauge action described as above. Drinfeld and Sokolov proved that \eqref{central} with $s=1$ can be reduced to $C^{\infty}(S^1, \b)/C^\infty(S^1, \N)$, and the resulting Poisson bracket is equivalent to the bracket previously described on the space of scalar operators. (One can reduce for any value of $s\ne 0$, all brackets are equivalent.)

The author of \cite{MB2} took a different approach. She identified the canonical form \eqref{can} as defining the Maurer-Cartan equation for projective curves, with $v_i$ its projective invariants. The moving frame associated to a curve would be $\rho(x) \in \SL(m)$, solution of $\L^{can}(\rho) = 0$, and the curve would be $\mu(x)$, the projectivization of $\gamma(x) = \rho(x) e_1$. Using the fact that that $\RP^{m-1} = \PSL(m)/H$, with $H$ the isotropic subgroup of the origin, defined by matrices of the form 
\[
\begin{pmatrix} \ast_1 & \ast\\ 0&\ast_{m-1}\end{pmatrix}\in \SL(m)
\]
she proved that by substituting $\b$ with $\sl(m)$, and $\N$ by $H$, and using the same action and central extension, one could show two facts: (a) the canonical forms \eqref{can} define a section of the new quotient; and (b) the Poisson structure \eqref{central} reduces to this new quotient and produces a Poisson bracket on the space of equivalence classes of projective curves under the action of the projective group. The reduced Poisson bracket was equal to that of Drinfeld and Sokolov. She also found a straightforward way to identify the curve evolutions
\[
\gamma_ t = X^f(\gamma)
\]
inducing the $f$-Hamiltonian flow on the invariants $v_i$. (These curve evolutions are uniquely defined on the lifts of projective evolutions to $\RR^m$.) The authors of \cite{CaliniIveyMB} defined the reduced Poisson bracket as pre-symplectic forms in the space of projective vector fields of curves in the projective plane. The authors of \cite{Terng19} described these pre-symplectic forms for general dimensions.

%\ai{Shall we describe the whole construction in the same way as for the scalar case above? That is, start with classical DS, say that it does not discretize directly, say that it can be viewed differently (by enlarging both the space and the group that is acting?), and that different perspective admits a discretization. We can then state Definition 1.2 of a priori another bracket on the same space. Then we can state a theorem that these two brackets are the same. } \textcolor{magenta}{From Gloria: that is a good idea, yes I can do that.}
 This geometric approach has a straightforward discretization. As in the continuous case, the authors of \cite{beffa2013hamiltonian} used the representation of $\RP^m = \PSL(m)/H$ as homogeneous space, and envisioned the space of classes of {\it projective twisted $N$-polygons} as a quotient $\PSL(m)^N/H^N$, with $H^N$ acting on $\PSL(m)^N$ by right discrete gauges.  Recall that a \textit{twisted $N$-gon} in $\RP^{m-1}$ is a bi-infinite sequence $\mu_n \in \RP^{m-1}$, $n\in \Z$, satisfying $\mu_{n+N} = \phi(\mu_n)$ for a certain projective transformation $\phi$ called {\it the monodromy}, a certain period $N\in \Z$, and any $n$. The right (respectively left) discrete gauge action is defined as a map $H^N\times G^N \to G^N$ given by
 \[
 (h_i, g_i) \mapsto (h_{i+1} g_i h_i^{-1}), \quad\quad (\text{respectively}~ (h_i^{-1} g_i h_{i+1})).
 \]
 Discrete gauge actions are Poisson maps for the well-known {\it twisted bracket}, originally defined by Semenov-Tian-Shansky in  \cite{semenov85}.  The authors of \cite{beffa2013hamiltonian} showed that Semenov-Tian-Shansky's bracket can be reduced to $\PSL(m)^N/H^N$,  defining a Poisson structure on the set of projective equivalence classes of twisted polygons. The reduction was a Hamiltonian structure for generalized Boussinesq lattice system of equations, and the authors found a second Hamiltonian structure and a master symmetry for the system.  
 
 Furthermore, the authors of \cite{beffa2013hamiltonian} showed the existence of a local vector field in the space of projective twisted polygons 
 \[
 \gamma_t = X^f(\gamma)
 \]
 with $\gamma$ being a uniquely chosen lift to $\RR^m$, inducing a reduced Hamiltonian evolution on the flow of equivalence classes, with Hamiltonian function $f$. These vector fields will be central to our main proof in the fourth section.
 
% In a second paper \cite{beffa2020compatible}, the authors proved that the two brackets were indeed compatible, and they conjectured the existence of a second completely integrable system, bi-Hamiltonian with respect to the brackets and linked to the kernel of  the companion bracket. This was proved by finding lifts of both brackets to a symplectic and a pre-symplectic form on the space of projectively invariant vector fields on the manifold of {\it twisted polygons} with period $N$ in $\RP^{m-1}$. The forms were connected to the brackets through properly chosen unique vector fields $X^f$ associated to an invariant Hamiltonian $f$, a connection that made the proof of the compatibility straightforward.

In our last chapter we prove that these two constructions produce the same Hamiltonian structure. The connection of the scalar case in the first half of the paper to the discretization of the Drinfeld-Sokolov reduction in the second half is readily provided by projective polygons: as we described above,
geometrically, the quotient of the space of $N$-periodic order $m$ difference operators by left and right multiplication is the set of projective equivalence classes of twisted polygons. Thus, the lattice $W_m$-algebra defined in the first half of the paper can be viewed as a Poisson structure on projective equivalence classes of twisted polygons. The results in our last chapter show that, given an invariant Hamiltonian $f$, one can readily identify a unique invariant polygonal vector field $Y^f$ directly from the scalar equations of the Poisson structure. This vector field induces the reduced Hamiltonian evolution with Hamiltonian function $f$ on monic operators of the form \eqref{redmonic}. If $X^f$ is the vector field above, the main theorem of our paper states
\begin{theorem}
If $X^f$ is defined by the reduction in \cite{beffa2013hamiltonian} and $Y^f$ is defined by the the discrete $W_m$-algebras, then, 
\[
X^f = Y^f.
\]
\end{theorem}

%Finally, we show that the elements in the kernel of the companion bracket which were conjectured to be seeds for an additional integrable system coincide with the lift of the Hamiltonians $H_1(M) = \sum_{n=1}^N \tr(M^{1/m}_n)$ and $H_2(M) = \sum_{n=1}^N \tr(M^{2/m}_n)$, where $M_n = a_0+a_1\T+\dots+a_n^{m-1}\T^{m-1}-\T^m$. {\bf I need to make this second part clearer, this is just the content.}

\section{A discretization of the Adler-Gelfand-Dickey bracket: Lattice $W_m$-algebras}\label{sec:scalar}
\subsection{Lattice $W_m$-algebras as polygon spaces}
We define a \textit{twisted $N$-gon} in $\P^{m-1}$
as a bi-infinite sequence of points $(p_i \in \P^{m-1})_{i \in \Z}$, along with
a projective transformation $\phi \in \PGL(m)$, called \textit{the monodromy}, such that
$p_{i+N} = \phi(p_i)$ for all $i$. We also assume that for every $i \in \Z$ the points $p_i, \dots, p_{i+m-1}$ do not belong to any proper projective subspace $\P^d \subset \P^{m-1}$. In other words, any lifts of these points to vectors in $\RR^m$ span $\RR^m$. Polygons with these property are sometimes called \textit{non-degenerate}. All polygons we consider have this property, so we refer to non-degenerate polygons simply as polygons. 
% \textcolor{red}{Gloria: Since \P^{m-1}$ is not a vector space, I am not sure this is clear. What about clarifying by adding "meaning that any lifts to $\RR^{m-1}$ will span $\RR^{m-1}$. Also, I think the literature calls these nondegenerate polygons}.

Denote the space of twisted $N$-gons  in $\P^{m-1}$ satisfying this condition by $\widetilde W_{m}$ (as is usually done when dealing with $W_m$-algebras, we suppress the dependence on $N$). The space $W_m$ is then defined as the quotient of $\widetilde W_m$ by the natural action of $\PGL(m)$. %These spaces contain their analogs $\widetilde \T_n$ and $\T_n$ defined in Section 2.2 as open dense subsets.

We aim to construct a Poisson structure on the space  $W_m$. To that end, we describe the latter space in terms of order $m$ difference operators. We say that a linear operator on the space $\RR^\infty$ of bi-infinite sequences is \textit{an upper-triangular difference operator of order $m$} if it has the form
\begin{equation}\label{app:operator}
 D =\sum_{i = 0}^m a^i \T^i,
\end{equation} 
where $a^i \in \RR^\infty$ are bi-infinite sequences acting on $\RR^\infty$ by term-wise multiplication, while $\T$ is the \textit{left} shift operator $(\T \xi)_j = \xi_{j+1} $ (the term \textit{upper-triangular} is used to distinguish such operators from those which may also contain terms of negative power in $\T$). The operator  \eqref{app:operator} is called \textit{$N$-periodic} if all its coefficients $a^i$ are $N$-periodic sequences, which is equivalent to saying that $\mathcal D T^N = T^N \mathcal D$.
We denote the space of $N$-periodic upper-triangular difference operators of order $m$ by $\DO{N}{m}$. 
The operator \eqref{app:operator} is called \textit{properly bounded} if its coefficient $a^0$ of lowest power in $\T$ and the coefficient $a^m$ of highest power in $\T$ are sequences of non-zero numbers. It is easy to see that the kernel of any properly bounded upper-triangular difference operator of order $m$  has dimension $m$. The space of properly bounded $N$-periodic upper-triangular difference operators of order $m$ will be denoted by $\PBDO{N}{m}$.

\begin{proposition}\label{app:identification}
There is a natural one-to-one correspondence between the following two sets: \begin{enumerate} \item The space $W_m$ of projective equivalence class of twisted $N$-gons in $\P^{m-1}$. \item The space $\PBDO{N}{m}$ of properly bounded $n$-periodic order $m$ upper-triangular difference operators, modulo the action
\begin{equation}\label{app:action}
\mathcal D \mapsto \alpha \mathcal D \beta^{-1},
\end{equation}
where $\alpha, \beta$ are $N$-quasi-periodic sequences of non-zero numbers with the same monodromy (i.e. there exists $z \in \RR^*$ such that $\alpha_{i + N}/\alpha_i = \beta_{i+N}/ \beta_i = z$ for every $i \in \Z$).
\end{enumerate}
\end{proposition}
%\begin{remark}
%One can also show that this correspondence between classes of operators and classes of polygons is a homeomorphism with respect to quotient topologies on both sets. Furthermore, formula~\eqref{app:projformulas} below shows that the space $\T_n'$ admits a dense chart in which the projection $\PBDO{n}{2} \to \T_n'$ is smooth.
%\end{remark}
\begin{proof}[Proof sketch]

The vertices of the polygon in $\P^{m-1}$ associated with a difference operator $\mathcal D \in \PBDO{N}{m}$ are given by
$v_i := (\xi^1_i : \ldots : \xi^m_i),$ where $\xi^1, \dots, \xi^m$ is any basis in the kernel of $\mathcal D$. Details of the proof for $m = 3$ can be found in \cite{izosimov2019pentagram}. The general case is completely analogous.
 \end{proof}
 \begin{remark}
 The above proposition is a version of a standard result on the relation between polygons and difference operators, see e.g. \cite[Proposition 4.1]{ovsienko2010pentagram}. A distinctive feature of our approach is the description of the moduli space of twisted polygons as a quotient, as opposed to a subspace, of the space of difference operators.
 \end{remark}
What we will do next is construct a Poisson structure on $N$-periodic order $m$ upper-triangular difference operators which is invariant under the action \eqref{app:action}. This will allow us to obtain a Poisson structure on $W_m$ by means of reduction. We call the space $W_m$, endowed with a so-obtained Poisson structure, a \textit{lattice $W_m$-algebra}. In fact, we will construct a Poisson structure which is defined on \textit{all} difference operators  (and on an even bigger space of \textit{pseudo-difference operators}) and then show that upper-triangular difference operators of fixed order form a Poisson submanifold. 

%In addition to that, we will show that the Poisson structure on difference operators can be restricted to those operators whose monodromy is a scalar matrix. 
%
%Recall that the \textit{monodromy} of an $n$-periodic difference operator $\mathcal D$ is defined as the restriction of the operator $T^n$ to $\Ker \mathcal D$. This notion is closely related to the monodromy of a twisted polygon. Namely, as can be seen from the proof of Proposition \ref{app:identification}, if $P$ is the twisted polygon associated with an $n$-periodic operator $\mathcal D$, then the monodromy of $P$ is equal to the projectivization of $M^*$, where $M:=T^n\vert_{\Ker \mathcal D}$ is the monodromy of $\mathcal D$. So, difference operators corresponding to closed polygons are exactly those whose monodromy is scalar. This will allow us to show that closed polygons form a Poisson submanifold of $\T_n'$.

\subsection{Poisson geometry of pseudo-difference operators and $W_m$-algebras}
We define an $N$-periodic \textit{pseudo-difference operator} as an expression of the form
\begin{equation}\label{app:psido}
\sum_{i = -\infty}^k a^i \T^i,
\end{equation}
where $\T$, as above, is the left shift operator on bi-infinite sequences, while each $a^i \in \RR^\infty$ is a bi-infinite sequence. Such an expression can be regarded either as a formal sum, or as an actual operator acting on the space $\{\xi \in \RR^\infty \mid \exists\, j \in \Z : \xi_i = 0 \, \forall \, i < j\}$ of ``semi-infinite'' sequences.\par
We will denote the set of $N$-periodic pseudo-difference operators by $\PSIDO{N}$. This set is an associative algebra. Moreover, almost every pseudo-difference operator is invertible. A sufficient condition for \eqref{app:psido} to be invertible is $a^k_i \neq 0$ for all $i \in \Z$. 
We will denote the set of invertible $N$-periodic pseudo-difference operators by $\IPSIDO{N}$. This is a group with respect to multiplication. At least formally, one can regard it as an infinite-dimensional Lie group. We will not discuss here how to define endow $\IPSIDO{N}$ with a manifold structure. 

\begin{remark}\label{app:loopgroup}
Consider the algebra of general (not necessarily upper-triangular) $N$-periodic {difference operators}, i.e. finite sums of the form
$
\sum a^i \T^i,
$
while each $a^i \in \RR^\infty$ is an $N$-periodic sequence. This algebra is isomorphic to the algebra $\gl(N) \otimes \RR[z,z^{-1}]$ of $\gl(N)$-valued Laurent polynomials in one variable~$z$. Indeed, consider the natural action of $N$-periodic difference operators on the space $V_z$ of all $N$-quasi-periodic sequences with monodromy $z$, i.e. sequences $\xi \in \RR^\infty$ that satisfy $\xi_{i+N} = z\xi_i$. This gives a $1$-parametric family $\rho_z$ of $N$-dimensional representations of the algebra of difference operators.
In each of the spaces $V_z$, we take a basis $\xi^1, \dots, \xi^N$ determined by the condition $\xi^i_j= \delta^i_j$ for $i,j = 1, \dots, N$. Written in this basis, the representation $\rho_z$ takes an $N$-periodic sequence $a$ (viewed as an order zero difference operator) to a diagonal matrix with entries $a_1, \dots, a_N$, while the shift operator $\T$ becomes the matrix $\sum_{i=1}^{n-1} {E}_{i,i+1} + z{E}_{N,1}$.
Therefore, since the algebra of difference operators is generated by order zero operators, $\T$, and $\T^{-1}$, it follows that $\rho_z$ can be viewed as a homomorphism of difference operators into
 $ \mathfrak{gl}(N) \otimes \RR[z, z^{-1}]$. Furthermore, it is easy to verify that this homomorphism is a bijection, and hence an isomorphism. Moreover, it is easy to see that this identification naturally extends to an isomorphism between the algebra of $N$-periodic pseudo-difference operators, and the algebra of matrices over the field $$\RR[z, z^{-1}]] := \left\{\sum_{i =-\infty}^k \alpha_i z^i \mid \alpha_i \in \RR\right\}$$ of semi-infinite Laurent series. Under this isomorphism, the group $\IPSIDO{N}$ of invertible pseudo-difference operators is identified with the group of matrices over Laurent series with non-vanishing determinant (this group is one of the versions of the \textit{loop group} of $\GL(N)$). 
\end{remark}
\begin{proposition}\label{app:mainprop}
There exists a natural Poisson structure $\pi$ on the group $\IPSIDO{N}$ of $N$-periodic invertible pseudo-difference operators. This structure has the following properties:
\begin{enumerate}
\item[1.]  It is multiplicative, in the sense that the group multiplication is a Poisson map. In other words, the group $\IPSIDO{N}$, together with the structure $\pi$, is a \textit{Poisson-Lie group}.
\item[2.] The subset $\IDO{N}{k} := \IPSIDO{N} \cap \DO{N}{k}$ of order $k$ invertible upper-triangular difference operators is a Poisson submanifold.
\item[3.]The Poisson structure $\pi$ vanishes on the submanifold  $\IDO{N}{0}$ of invertible order zero operators.
\item[4.] The Poisson structure $\pi$ is invariant under an automorphism of $\IPSIDO{N}$ given by conjugation $\mathcal D \to \alpha \mathcal D \alpha^{-1}$ with quasi-periodic $\alpha \in \RR^\infty$.
%\begin{equation}\label{app:aut}
%\sigma_t\left(\sum_{i = -\infty}^k a_i T^i\right) = \sum_{i = -\infty}^k a_i t^iT^i.
%\end{equation}
%\item If $\mathcal D$ belongs to the center of $\IPSIDO{n}$, then the Poisson structure $\pi$ vanishes at $\mathcal D$.
\end{enumerate}
\end{proposition}
\begin{remark}
As explained in Remark \ref{app:mainprop}, the group  $\IPSIDO{N}$ is isomorphic to a version of the loop group of $\GL(N)$. In the loop group language, the Poisson structure $\pi$ is well-known: it is the one associated with the so-called \textit{trigonometric $r$-matrix}. Here we will provide a construction of this Poisson structure which does not appeal to the loop group formalism. In fact, the language of (pseudo)difference operators seems to be more natural when dealing with the trigonometric $r$-matrix. \end{remark}
%\begin{remark}
%Our Poisson structure on pseudo-difference operators can be viewed as a discretization of the Poisson-Lie (=Gelfand-Dickey) structure on pseudodifferential operators \cite{khesin1995poisson}.
%\end{remark}
\begin{remark}
As is usually the case in infinite-dimensional Poisson geometry, the Poisson bracket $\pi$ on $\IPSIDO{N}$ is only defined for a certain relatively small subalgebra of functions. The good news is that $\pi$ is well-defined for coordinate functions, which guarantees that the restriction of $\pi$ to the finite-dimensional submanifold of fixed order upper-triangular difference operators is a genuine Poisson bracket (i.e. the bracket of any two smooth functions is well-defined).
\end{remark}
We will prove Proposition \ref{app:mainprop} in  Section \ref{app:mainproof}, after a brief general discussion of Poisson-Lie groups. What we will do now is use Proposition \ref{app:mainprop} to construct a Poisson bracket on the space $W_m$ of (equivalence classes of) twisted polygons. %Furthermore, we will show that this bracket restricts to the submanifold of closed polygons.
\begin{proposition}\label{app:tnpoisson} The Poisson structure $\pi$ from Proposition \ref{app:mainprop} induces a Poisson structure on the space $W_m$ of projective equivalence classes of twisted $N$-gons in $\P^{m-1}$.\end{proposition}
\begin{proof}
According to the second statement of Proposition \ref{app:mainprop}, the space $\IDO{N}{m}$ of invertible $N$-periodic order $m$ upper-triangular difference operators is a Poisson submanifold of $\IPSIDO{N}$. Furthermore, the space  $\PBDO{N}{m}$ of properly bounded $N$-periodic order $m$ upper-triangular difference operators is a subset of $\IDO{N}{m}$ (because non-vanishing of the leading term $a^m$ in \eqref{app:operator} guarantees invertibility), and moreover it is clearly an open (and dense) subset. So, since  $\PBDO{N}{m}$ is an open subset in a Poisson submanifold, it is also a  a Poisson submanifold.  Thus, to prove that the Poisson structure $\pi$ descends to $W_m$, it remains to show that $\pi$ is invariant under the action \eqref{app:action} (and apply Proposition \ref{app:identification}). To that end, we describe that action slightly differently. Namely, notice that since the sequences $\alpha$, $\beta$ in \eqref{app:action} have the same monodromy $z$, one can represent the  action \eqref{app:action} as conjugation by $\beta$ followed by left multiplication with a periodic sequence $\alpha \beta^{-1}$. Furthermore, conjugation by $\beta$ is a Poisson map by the fourth statement of Proposition \ref{app:mainprop}. So it remains to show that left action of the group $\IDO{N}{0}$ of periodic non-vanishing sequences is Poisson. To that end, observe that for any $\alpha \in \IDO{N}{0}$, its action on $\IPSIDO{N}$ is a composition of the embedding $\IPSIDO{N} \to \IPSIDO{N} \times \IPSIDO{N}$ given by $\mathcal D \mapsto (\alpha, \mathcal D)$ and the multiplication map $\IPSIDO{N} \times \IPSIDO{N} \to \IPSIDO{N}$. These two maps are Poisson by, respectively, the third and the first statement of Proposition \ref{app:mainprop}. So, the left action of $\IDO{N}{0}$ on $\IPSIDO{N}$, and hence on $\PBDO{N}{m}$, is Poisson as well. Therefore, the action \eqref{app:action} is Poisson, and the Poisson structure $\pi$ descends to the space $W_m$ of projective equivalence classes of twisted polygons.
\end{proof}

\begin{definition}
The space $W_m$ of projective equivalence classes of twisted $N$-gons in $\P^{m-1}$ endowed with a Poisson structure of Proposition \ref{app:tnpoisson} is called a \textit{lattice $W_m$-algebra}.
\end{definition}
\subsection{Poisson-Lie groups}
This section is a brief introduction to the theory of Poisson-Lie groups. Our terminology mainly follows that of \cite{reiman2003integrable}. \par
Recall that a Lie group $G$ endowed with a Poisson structure $\pi$ is called a \textit{Poisson-Lie group} if $\pi$ is \textit{multiplicative}, i.e. if the multiplication $G \times G \to G$ is a Poisson map. Assume that $G$ is a Poisson-Lie group, and let $\g$ be its Lie algebra. Then, using either left or right trivialization of the tangent bundle of $G$, one can identify the bivector field $\pi$ with a map $G \to \g \wedge \g$. Furthermore, one can show that multiplicativity of $\pi$ is equivalent to this map being a cocycle on $G$ with respect to the adjoint representation of $G$ on $ \g \wedge \g$. If, moreover, that cocycle is a coboundary, then $G$ is called a \textit{coboundary Poisson-Lie group}. A Poisson-Lie group $G$ is coboundary  if and only if there exists an element $r \in \g \wedge \g$, called the \textit{classical $r$-matrix}, such that the Poisson tensor $\pi$ at every point $g \in G$ is given by
\begin{equation}\label{app:cobound}
\pi_g = (\lambda_g)_*r - (\rho_g)_*r \quad \forall\, g \in G,
\end{equation}
where $\lambda_g$ and $\rho_g$ are, respectively, the left and right translations by $g$. 
Note that although the bivector~\eqref{app:cobound} is automatically multiplicative (since any coboundary is a cocycle), it does not need to satisfy the Jacobi identity. The necessary  and sufficient condition for \eqref{app:cobound} to satisfy the Jacobi identity is a rather complicated equation in terms of $r$, so it is usually replaced by simpler sufficient conditions, such as the \textit{classical Yang-Baxter equation}.  In our case, however,  the relevant condition is the \textit{modified Yang-Baxter equation}. We will formulate it under the assumption that the Lie algebra $\g$ is endowed with an $\mathrm{Ad}$-invariant inner product (i.e. an inner product which is invariant under the adjoint action of $G$ on $\g$), in which case one can identify the bivector $r \in \g \wedge \g$ with a skew-symmetric operator $ \g \to \g$. Then the modified Yang-Baxter equation for $r$ reads
\begin{equation}\label{app:mybe}
[rx, r y] - r[rx, y] - r[x, r y] = c [x,  y] \quad \forall\, x,y \in \g,
\end{equation}
where $c \in \RR$ is a constant not depending on $x, y$.
It is well-known that this equation implies (although is not necessary for) the Jacobi identity for~\eqref{app:cobound}. If the Lie algebra of a coboundary Poisson-Lie group $G$ is endowed with an invariant inner product, and the corresponding $r$-matrix satisfies the modified Yang-Baxter equation \eqref{app:mybe}, then $G$ is called \textit{factorizable}. 
\begin{example}
Let $\g$ be a Lie algebra endowed with an invariant inner product $(\,,)$. Assume also that $\g$, as a vector space, can be written as a direct sum of two subalgebras $\g_+$ and $\g_-$, both of which are isotropic with respect to the inner product  $(\,,)$. Then $(\g, \g_+, \g_-)$ is called a \textit{Manin triple}. Any Manin triple gives rise to a factorizable Poisson-Lie structure on the group $G$ associated with the algebra $\g$. The corresponding $r$-matrix $r \colon \g \to \g$, satisfying the modified Yang-Baxter equation~\eqref{app:mybe}, is given by $r:= \frac{1}{2}(p_+ - p_-)$, where $p_\pm$ are projectors $\g \to \g_\pm$.
\end{example}
\begin{example}\label{app:gmt}
This example is a generalization of the previous one. Let $\g$ be a Lie algebra endowed with an invariant inner product. Assume also that $\g$, as a vector space, can be written as a direct sum of three subalgebras $\g_+$, $\g_0$, and $\g_-$, such that $[\g_0, \g_\pm] = \g_\pm$, the subalgebras $\g_\pm$ are isotropic, and $\g_0$ is orthogonal to both $\g_+$ and $\g_-$.
 Then $r:= \frac{1}{2}(p_+ - p_-)$ satisfies the modified Yang-Baxter equation, thus turning the group $G$ of the Lie algebra $\g$ into a factorizable Poisson-Lie group. Later on we will use the term $r$-matrix for $r+\frac12 \mathrm{id}$, using the notation $\widehat r$ instead.

\end{example}
We will define a Poisson structure on the group $\IPSIDO{N}$ of $N$-periodic invertible pseudo-difference operators using the construction described in Example \ref{app:gmt}. But before we do that, let us notice that the group $\IPSIDO{N}$ can be embedded, as an open dense subset, into the associative algebra $\PSIDO{N}$ of all (not necessarily invertible) $N$-periodic pseudo-difference operators. This simplifies things a lot, so we would like to investigate this situation in more detail.
Assume that a Lie group $G$ is embedded, as an open subset, into an associative algebra $A$ (in a typical situation $G$ coincides with the group of invertible elements in $A$). Then the Lie algebra of $G$ (and, more generally, the tangent space to $G$ at any point) can be naturally identified with $A$. Assume also that $A$ is endowed with an invariant inner product, which in the context of associative algebras means that $(xy,z) = (x,yz)$ for any $x,y,z \in A$ (in particular, this inner product is invariant with respect to the adjoint action of $G \subset A$ on $A$). Furthermore, assume that $r \colon A \to A$ is a skew-symmetric operator satisfying the modified Yang-Baxter equation~\eqref{app:mybe}. Then $G$ carries a structure of a factorizable Poisson-Lie group. Identifying the cotangent space $T_g^*G$ with the tangent space $T_g G = A$ by means of the invariant inner product, one can then rewrite formula~\eqref{app:cobound} for the corresponding Poisson tensor on $G$ as
\begin{equation}\label{app:gd}
\pi_g(x,y) = (r(xg), yg) - (r(gx), gy) \quad \forall\, g \in G, x,y \in A.
\end{equation}
%\ai{Here we can give a formula for the bracket of functions in terms of left and right gradients if needed.} \textcolor{red}{I think that will be helpful, that way we can start connecting both approaches} \ai{Is formula (11) from \url{https://arxiv.org/pdf/1803.00726.pdf} what you have in mind? So left (right) gradient in this setting is just the gradient multiplied on the left (respectively, right) by $g$? In that case, perhaps one can just write those gradients as $g (\nabla f)$ and $(\nabla f) g$ without introducing the notions of left and right gradients? I can definitely see that these gradients are necessary for a general Lie group, where there are two ways to trivialize the tangent bundle, but in the associative algebra setting it is already trivial, so are these necessary?}
\begin{remark}
The reader might notice that the right-hand side of \eqref{app:gd} is actually defined for every $g \in A$ (invertibility of $g$ is not necessary). So, this formula may be used to define a Poisson bracket on the whole of $A$. This bracket is sometimes called the \textit{(second) Gelfand-Dickey bracket} on the associative algebra~$A$.
\end{remark}
Before we apply this construction to the group $\IPSIDO{N}$, let us mention the following general fact about coboundary Poisson-Lie groups, which we will need to prove Proposition \ref{app:mainproof}:
\begin{proposition}\label{app:autprop}
Let $\sigma \colon G \to G$ be an automorphism of a coboundary Poisson-Lie group. Assume that the differential of $\sigma$ at the identity preserves the $r$-matrix (viewed as an element of $\g \wedge \g$) . Then $\sigma$ is a Poisson map.
\end{proposition}
\begin{proof}
Since $\sigma$ is an automorphism, we have
$
 \lambda_{\sigma(g)} = \sigma  \lambda_g \sigma^{-1}$ and   $\rho_{\sigma(g)} =  \sigma  \rho_g \sigma^{-1},
$
so 
$$
\pi_{\sigma(g)} = (\lambda_{\sigma(g)})_*r - (\rho_{\sigma(g)})_*r = \sigma_* (\lambda_g)_* (\sigma^{-1})_*r -  \sigma_* (\rho_g)_* (\sigma^{-1})_*r.
$$
Since $\sigma$ preserves the $r$-matrix, the latter expression can be rewritten as
$
 \sigma_* (\lambda_g)_* r -  \sigma_* (\rho_g)_*r = \sigma_*\pi_g.
$
So, $\pi_{\sigma(g)} = \sigma_*\pi_g$, which means that $\sigma$ is a Poisson map.
\end{proof}
%\begin{proposition}\label{app:center}
%Assume that $g$ is a central element of $G$. Then any coboundary Poisson-Lie bracket on $G$ vanishes at $g$.
%\end{proposition}
%\begin{proof}
%Since $g$ is central, we have $\lambda_g = \rho_g$, so the right-hand side of \eqref{app:cobound} vanishes.
%\end{proof}
\subsection{Construction of the Poisson bracket on pseudo-difference operators}\label{app:mainproof}
\begin{proof}[Proof of Proposition\ref{app:mainprop}]
To define a Poisson structure on the group $\IPSIDO{N}$ of $N$-periodic invertible pseudo-difference operators, we will use the construction described in Example \ref{app:gmt}. The Lie algebra of the group $\IPSIDO{N}$ is the space $\PSIDO{N}$ of all $N$-periodic  pseudo-difference operators. That is actually an associative algebra in which $\IPSIDO{N}$ is embedded as the set of invertible elements. So, after we describe an invariant inner product and an $r$-matrix, we will be able to use formula \eqref{app:gd}. The invariant inner product is given by
\begin{equation}\label{app:ip}
(\mathcal D_1, \mathcal D_2) = \Tr \mathcal D_1\mathcal D_2 \quad \forall \, \mathcal D_1, \mathcal D_2 \in \PSIDO{N},
\end{equation}
where the \textit{trace} of an $N$-periodic pseudo-difference operator $\mathcal D$ is defined by
$$
\Tr \left( \sum_{i = -\infty}^k a^i \T^i\right) := \sum_{j=1}^N a^{0}_j.
$$
This product is clearly non-degenerate and invariant in the associative algebra sense. Furthermore, one can verify explicitly that $ \Tr \mathcal D_1\mathcal D_2 =  \Tr \mathcal D_2\mathcal D_1$, so the inner product~\eqref{app:ip} is symmetric. Alternatively, this can be showed by using the isomorphism of $\PSIDO{N}$ and the loop algebra of $\gl(N)$ (see Remark \ref{app:loopgroup}). In the loop algebra language, the trace of an operator can be written as
$$
\Tr \mathcal D = \mathrm{res}_{z=0}z^{-1} \mathrm{tr} \, \mathcal D,
$$
where $\mathrm{tr}$ is the matrix trace.\par
Now, we represent $\PSIDO{N}$ as a sum of three subalgebras $\g_-$, $\g_0$, $\g_+$. Namely, we define
$$
\g_- := \left\{ \sum_{i = -\infty}^{-1} a^i \T^i\right\}$$ as the subalgebra of operators which only contain terms of negative power in $\T$, $\g_0   : =  \DO{N}{0}$ as the subalgebra of order zero difference operators (i.e. $N$-periodic bi-infinite sequences), and $$\g_+ := \left\{ \sum_{i = 1}^{k} a^i \T^i\right\}
$$ as the subalgebra of operators which only contain terms of positive power in $\T$. This decomposition clearly satisfies all the requirements of Example \ref{app:gmt}, so we get an $r$-matrix $r: = \frac{1}{2}(p_+ - p_-)$, and hence a factorizable Poisson-Lie structure on $\IPSIDO{N}$. This proves the first statement of Proposition \ref{app:mainprop}. To prove the second statement, we use formula~\eqref{app:gd}. From that formula it follows that, when viewed as map $\PSIDO{N} \to \PSIDO{N}$, the Poisson tensor $\pi_{\mathcal D}$ (where $\mathcal D \in \IPSIDO{N}$) reads
\begin{equation}\label{app:ho}
\pi_\mathcal D(\mathcal Q) = \mathcal D r(\mathcal Q\mathcal D) - r(\mathcal D \mathcal Q)\mathcal D.
\end{equation}
To show that the subset $\IDO{N}{m} \subset \IPSIDO{N}$ of order $m$ invertible upper-triangular difference operators is a Poisson submanifold, one needs to prove that for $\mathcal D \in \IDO{N}{m}$ the image of the Poisson tensor \eqref{app:ho} belongs to the tangent space to $\IDO{N}{m}$ at $\mathcal D$. The latter is the space $\DO{N}{m}$ of all $N$-periodic order $m$ upper-triangular difference operators, so we need to show that the right-hand side of \eqref{app:ho} is an order $m$ upper-triangular difference operator whenever $\mathcal D$ is an order $m$ upper-triangular difference operator. To that end, notice that the right-hand side of  \eqref{app:ho} stays the same if $r$ is replaced by $r \pm \frac{1}{2}\Id$. But the image of $r + \frac{1}{2} \Id =  p_+ + \frac{1}{2} p_0$  (where $p_0$ is the projector to $\g_0$) is contained in $\g_0 + \g_+$, so the right hand-side of~\eqref{app:ho} is in $\g_0 + \g_+$ as well. At the same time, the image of $r - \frac{1}{2}\Id = - p_- - \frac{1}{2} p_0$ is the space $\g_0 + \g_-$,  so the right hand-side of~\eqref{app:ho} cannot contain terms whose degree in $\T$ exceeds $m$. Thus, the right-hand side of~\eqref{app:ho} is indeed an upper-triangular difference operator of order at most $m$, as desired.\par
To prove the third statement of Proposition \ref{app:mainprop}, notice that for an order zero operator $\mathcal D$ one has
$
r(\mathcal Q\mathcal D) = r(\mathcal Q)\mathcal D$ and  $r(\mathcal D \mathcal Q) = \mathcal D r(\mathcal Q)$,
so from \eqref{app:ho}  we get that the Poisson tensor $\pi$ vanishes at $\mathcal D$. To prove the fourth statement, we need to show that conjugation by a quasi-periodic sequence preserves the $r$-matrix (see Proposition~\ref{app:autprop}). But this follows from preservation of $\g_\pm$ and $\g_0$ along with the inner product. %Finally, the last statement of Proposition~\ref{app:mainprop} directly follows from Proposition~\ref{app:center}. 
So, Proposition~\ref{app:mainprop} is proved. 
\end{proof}
\subsection{$W_2$, Virasoro, and Volterra}\label{app:voltproof}

In this section we show that the lattice $W_2$-algebra as defined above coincides with the lattice Virasoro algebra of Faddeev-Takhtajan-Volkov \cite{volkov1988miura, faddeev2016liouville}. The corresponding Poisson structure also arises in integrable systems as the cubic Poisson bracket associated to the Volterra lattice~\cite{suris1999integrable}.

In order to compute the $W_2$ structure in coordinates, we use the geometric description of the space  $W_2$ as the space of projective equivalence classes of twisted polygons in $\RP^1$. Let $(p_i \in \RP^1)_{i \in \Z}$ be a twisted $N$-gon. Set
\begin{equation}\label{eq:crc}
x_i := [p_{i-1}, p_i, p_{i+1}, p_{i+2}],
\end{equation}
where the cross-ratio $[a,b,c,d]$ of four points $a,b,c,d \in \RP^1$ is defined by
$$
[a,b,c,d]:=\frac{(a-b)(c-d)}{(a-c)(b-d)}.
$$
Then the numbers $x_i$, subject to the periodicity condition $x_{i+N} = x_i$, define a coordinate chart on an open dense subset of $W_2$, see \cite{arnold2018cross}.

We will compute the $W_2$ structure in coordinates $x_i$. To that end, we first need to describe the restriction of the structure $\pi$ on $\IPSIDO{N}$ to the space of $n$-periodic second order upper-triangular difference operators. For notational simplicity, we will write such an operator as $\alpha + \beta \T + \gamma \T^2$. As coordinates on the space of such operators, one can take the entries $\alpha_i, \beta_i, \gamma_i$ of the $N$-periodic sequences $\alpha, \beta, \gamma$.
\begin{proposition} Assume that $N > 2$.
Then the restriction of the Poisson structure $\pi$ to the space of $n$-periodic second order upper-triangular difference operators is given by the following formulas:
\begin{align}\label{app:pbformulas}
\begin{gathered}
\{\alpha_i,\beta_{i-1}\} = -\frac{1}{2}\alpha_i \beta_{i-1}, \quad \{\alpha_i,\beta_i\} = \frac{1}{2}\alpha_i \beta_i,   \quad \{\alpha_i,\gamma_{i-2}\} = -\frac{1}{2}\alpha_i \gamma_{i-2},\\ \quad \{\alpha_i,\gamma_i\} = \frac{1}{2}\alpha_i \gamma_i, \quad
\{\beta_i, \beta_{i+1}\} = \alpha_{i+1}\gamma_i, \\ \{\beta_i, \gamma_{i-1}\} = -\frac{1}{2}\beta_i \gamma_{i-1}, \quad \{\beta_i, \gamma_i\} = \frac{1}{2}\beta_i\gamma_i,
\end{gathered}
\end{align}
while all other Poisson brackets vanish.
\end{proposition}
\begin{proof}
This is a straightforward computation. As an example, let us compute the bracket of $\alpha_i$ and $\gamma_j$. First, note that the differentials of these functions (identified with elements of the algebra $\PSIDO{N}$ by means of the invariant inner product) are given by $d\alpha_i = \delta_i$, $d\gamma_j = \T^{-2}\delta_j$, where $\delta_i$ is an $N$-periodic sequence whose terms are given by
\begin{align*}
\delta_{ij} := \left[\begin{aligned}
1\quad \mbox{ if } i = j \mod N,\\
0 \quad  \mbox{ if } i \neq j \mod N.
\end{aligned}\right.
\end{align*}
Then, using formula \eqref{app:gd}, we get the following formula for the bracket $\{\alpha_i, \gamma_j \} $ computed at $\mathcal D := \alpha + \beta \T + \gamma \T^2$:
\begin{gather*}
\{\alpha_i, \gamma_j \} = \Tr (r(\delta_i \mathcal D)\T^{-2}\delta_j \mathcal D - r(\mathcal D \delta_i)\mathcal D \T^{-2}\delta_j) \\ =  \Tr ((r - \frac{1}{2}\Id)(\delta_i \mathcal D)\T^{-2}\delta_j \mathcal D - (r - \frac{1}{2}\Id)(\mathcal D \delta_i)\mathcal D \T^{-2}\delta_j) 
\\= -\frac{1}{2}\Tr (\delta_i \alpha \T^{-2}\delta_j \mathcal D -  \delta_i \alpha\mathcal D \T^{-2}\delta_j)  = -\frac{1}{2}\Tr (\delta_i \alpha \T^{-2}\delta_j \gamma \T^2 -  \delta_i \alpha \gamma \delta_j)\\ =  -\frac{1}{2}\Tr (\delta_i \alpha \delta_{j+2} (\T^{-2}\gamma) -  \delta_i \alpha\gamma \delta_j) 
= \frac{1}{2}(-\delta_{i, j+2} \alpha_i \gamma_{i-2} + \delta_{ij}\alpha_i \gamma_i).
\end{gather*}
Assuming that $N > 2$, this gives   $\{\alpha_i,\gamma_{i-2}\} = -\frac{1}{2}\alpha_i \gamma_{i-2}$, $\{\alpha_i,\gamma_i\} = \frac{1}{2}\alpha_i \gamma_i$, while other Poisson brackets of the form $\{\alpha_i, \gamma_j \} $ vanish.
\end{proof}
We now describe the Poisson structure on the space $W_2$ of equivalence classes of twisted polygons. To that end, we obtain coordinate expressions for the projection $\PBDO{N}{2} \to W_2$. As coordinates on the space $\PBDO{N}{2} = \{\alpha + \beta \T + \gamma \T^2 \}$ we take the entries $\alpha_i, \beta_i, \gamma_i$ of the coefficients $\alpha, \beta, \gamma$, while as coordinates on $W_2$ we will use the cross-ratios $x_i$ defined by \eqref{eq:crc}. Note that the coordinates $x_i$ are only defined on a dense subset of $W_2$, so the projection $(\alpha_i, \beta_i, \gamma_i) \mapsto x_i$ will also be defined on an open dense subset. 
\begin{proposition}
The projection $\PBDO{N}{2} \to W_2$ is given by
\begin{equation}\label{app:projformulas}
x_i = \frac{\alpha_i \gamma_{i-1}}{\beta_{i-1}\beta_i}.
\end{equation}
\end{proposition}
\begin{proof}
By construction of the polygon $(p_i \in \RP^1)$ associated with a difference operator $\alpha + \beta \T + \gamma \T^2$ (see the proof of Proposition \ref{app:identification}), there exists a sequence $V$ of vectors $V_i \in \RR^2$ such that $V_i$ is the lift of $p_i$, and $V_i$'s satisfy the difference equation
\begin{equation}\label{app:de}
\alpha_i V_i + \beta_i V_{i+1} + \gamma_i V_{i+2} = 0.
\end{equation}
Then, since $V_i$'s are lifts of $p_i$'s, we have
$$
x_i = \frac{(p_{i-1} - p_i)(p_{i+1} - p_{i+2})}{(p_{i-1} - p_{i+1})(p_{i} - p_{i+2})} = \frac{\det(V_{i-1}, V_i)\det(V_{i+1}, V_{i+2})}{\det(V_{i-1}, V_{i+1})\det(V_{i}, V_{i+2})}. 
$$
Expressing $V_{i+2}$ from \eqref{app:de}, this can be rewritten as
\begin{align*}
x_i &= \frac{\det(V_{i-1}, V_i)\det(V_{i+1}, -\gamma_i^{-1}(\alpha_i V_i + \beta_i V_{i+1}) )}{\det(V_{i-1}, V_{i+1})\det(V_{i}, -\gamma_i^{-1}(\alpha_i V_i + \beta_i V_{i+1}) )} =  -\frac{\alpha_i \det(V_{i-1}, V_i)}{\beta_i\det(V_{i-1}, V_{i+1})} \\ &\qquad\quad= -\frac{\alpha_i \det(V_{i-1}, V_i)}{\beta_i\det(V_{i-1}, -\gamma_{i-1}^{-1}(\alpha_{i-1} V_{i-1} + \beta_{i-1} V_{i}))} = \frac{\alpha_i \gamma_{i-1}}{\beta_{i-1}\beta_i},
\end{align*}
q.e.d.
\end{proof}
\begin{proposition}\label{prop:w2vol}
Assume that $N > 2$. Then, in cross-ratio coordinates $x_i$, the $W_2$ structure is given by
\begin{equation}\label{eq:w2vol}
\{x_i, x_{i+1}\} = x_i x_{i+1}(x_i + x_{i+1} - 1), \quad
\{x_i, x_{i+2}\} = x_i x_{i+1}x_{i+2}.
\end{equation}
\end{proposition}
\begin{remark}
Up to sign, bracket \eqref{eq:w2vol} coincides with the cubic Poisson structure for the Volterra integrable system, see \cite[Section 10.1]{suris1999integrable}. Furthermore, rewritten in terms of the variables $s_i := 4x_i$, our bracket coincides (up to a constant factor) with the lattice Virasoro structure of \cite{volkov1988miura, faddeev2016liouville}.
\end{remark}

\begin{proof}[Proof of Proposition \ref{prop:w2vol}]
The proof is achieved by computing the Poisson brackets of $x_i$'s using their expressions~\eqref{app:projformulas} in terms of $\alpha_i, \beta_i, \gamma_i$, along with formulas \eqref{app:pbformulas} for pairwise Poisson brackets of $\alpha_i, \beta_i, \gamma_i$.  For example,
\begin{gather*}
\{ \log x_i, \log x_{i+1}\} = \{ \log \alpha_i  + \log \gamma_{i-1} - \log \beta_{i-1} - \log \beta_i,  \log \alpha_{i+1}  + \log \gamma_{i} - \log \beta_{i} - \log \beta_{i+1}\}
\\ = \{ \log \alpha_i  , \log \gamma_i \} -  \{ \log \alpha_i  , \log \beta_i \} + \{\log \gamma_{i-1},  \log \alpha_{i+1}\} - \{\log \gamma_{i-1},  \log \beta_{i}\} \\ +\, \{ \log \beta_{i-1}, \log \beta_i \}  - \{\log \beta_{i}, \log \alpha_{i+1}\} - \{\log \beta_{i}, \log \gamma_{i}\} + \{ \log \beta_{i}, \log \beta_{i+1} \} \\
= \frac{1}{2} - \frac{1}{2} + \frac{1}{2} - \frac{1}{2} + \,\frac{\alpha_i \gamma_{i-1}}{\beta_{i-1}\beta_i} - \frac{1}{2} - \frac{1}{2} + \,\frac{\alpha_{i+1} \gamma_{i}}{\beta_{i}\beta_{i+1}} = x_i + x_{i+1} - 1,
\end{gather*}
which implies the desired formula for $\{x_i, x_{i+1}\}$.
\end{proof}
\section{A discretization of the geometric Drinfeld-Sokolov reduction}
In \cite{beffa2013hamiltonian} the authors constructed what can be considered as a discretization of a reduction of Drinfeld-Sokolov type in the $G = \SL(m)$ case. The main Poisson bracket, known as the {\it twisted Poisson bracket},  was constructed by Semenov-Tian-Shansky in~\cite{semenov85} using the Poisson-Lie group theory described in our previous chapter. We refer the reader to the original paper for details of this construction, while we will give its definition below. If $\RP^{m-1} = \PSL(m)/H$ is the standard homogeneous representation of the projective space, the authors of \cite{beffa2013hamiltonian} showed that the twisted Poisson bracket, originally defined in $\SL(m)^N$ can be reduced to a quotient $\SL(m)^N/H^N$, with $H^N$ acting on $\SL(m)^N$ by (right) gauges. In this section we will show that this reduction also goes through when $G = \GL(m)$, for a properly chosen subgroup~$H$.
\subsection{Poisson reduction in the $\GL(m)$ case}
Consider $G = \GL(m)$, $\g = \gl(m)$, and endow $\g^N$ with the invariant inner product
\[\begin{matrix}
\g^N \times \g^N &\to& \RR\\
(\mu, \eta) &\mapsto& \langle \mu, \eta\rangle = \sum_{i=1}^N \tr(\mu_i\eta_i).
\end{matrix}
\]
Given $\F: G^N \to \RR$, we define the left and right gradients at $A\in G^N$ as $\nabla\F(A), \nabla'\F(A)\in \g^N$, given by the left and right translation of the differential, considered as an element of the cotangent, and identifying $\g^\ast$ with $\g$ using the invariant inner product. (In the previous section's language, they will play the role of $\mathcal{D}\mathcal{Q} $ and $\mathcal{Q} \mathcal{D}$ in \eqref{app:ho}.) That is
\[
\left.\frac{d}{d\epsilon}\right|_{\epsilon = 0}\F(e^{\epsilon \xi} A)= \langle\nabla\F(A) , \xi\rangle; \quad  \left.\frac{d}{d\epsilon}\right|_{\epsilon = 0}\F(Ae^{\epsilon \xi})= \langle\nabla'\F(A) , \xi\rangle.
\]

Let  
\[
\r = \sum_{i>j} E_{ij}\otimes E_{ji} + \frac 12\sum_r E_{rr}\otimes E_{rr}
\]
be the standard $r$-matrix for $G=\GL(m)$. Notice that $\r = r+\frac 12 \mathrm{Id}$ as defined in the previous section. 

Given $\F, \H$ smooth scalar-valued functions on $G^N$ and $A\in G^N,$ the \emph{twisted Poisson bracket} is given by \cite{frenkel1998drinfeld}:
\begin{equation}\label{twisted}
\begin{split}
\{\F, \H\}(A) & := \sum_{s=1}^N\r(\nabla_s \F \wedge \nabla_s \H) + \sum_{s=1}^N \r(\nabla_s' \F \wedge \nabla_s' \H)  \\ &- \sum_{s=1}^N \r\left( (\T\otimes 1)(\nabla_s' \F \otimes \nabla_s \H)\right) + \sum_{s=1}^N \r  \left((\T\otimes 1) (\nabla_s' \H \otimes \nabla_s \F)\right), 
\end{split}
\end{equation}
where $\xi\wedge \eta=\tfrac12( \xi\otimes\eta - \eta\otimes\xi)$ and $\T$ is the cyclic shift on $\g^N$.
Formula~\eqref{twisted} defines a Hamiltonian structure on $G^{N}$, as shown by Semenov-Tian-Shansky in \cite{semenov85}.  Moreover, the {\em right gauge action} of $G^{N}$ on itself 
\begin{equation}\label{gauge}
(g, A) \mapsto (\T g A g^{-1})
\end{equation}
is a Poisson map and its orbits coincide with the symplectic leaves \cite{frenkel1998drinfeld, semenov85}.

Let $H\subset \GL(m)$ be the subgroup defined by 
\[
H = \{ A\in G, ~~A e_1 =  e_1 \}.
\]

\begin{theorem}\label{reducedbracket}
The Poisson bracket (\ref{twisted}) reduces to the quotient $G^N/H^N$, where $H^N$ is acting on $G^N$ via the right gauge action. Note that $G^N/H^N$ is a manifold only at generic points. Here and in what follows by Poisson reduction to $G^N/H^N$ we mean reduction of an open dense subset of $G^N$, so that the quotient is smooth.
\end{theorem}
\begin{proof}
The proof is very similar to that in \cite{beffa2013hamiltonian} where the same theorem was proved for the case $G = \SL(m)$. Since the orbits of the action coincide with the symplectic leaves, the bracket will reduce whenever  $(\h^N)^0$ is a Lie subalgebra of $(\g^N)^\ast$, where $\h$ is the Lie algebra of $H$.
Recall that the linear bracket in a dual algebra is defined by the linearization at the identity $e\in G$ of the twisted Poisson bracket. That is
\[
[d_e\phi, d_e\varphi]_\ast = d_e\{\phi, \varphi\} \in \g^\ast.
\]
Since $\h$ is defined by matrices with zero first column, $\h^0$ is defined by matrices whose only nonzero entries are in the first row. We will first look for functions $\phi^i_r$ such that $d_e\phi^i_r$ generate $(\h^N)^0$. 

Indeed, let $A\in G^{N}$ be close enough to $e\in G^{N}$ so that $A = (A_r)$ can be factored as
\[
A_r = \begin{pmatrix} 1& q_r^T\\ 0&I_{m-1} \end{pmatrix} \begin{pmatrix}1&{\bf 0}^T \\ {\bf 0} & \Theta_r  \end{pmatrix}\begin{pmatrix} \theta_r&  {\bf 0}^T \\ \ell_r& I_{m-1} \\\end{pmatrix}
\]
which factors $H$ to the left as defined by $\ell_r = 0$ and $\theta_r = 1$, for all $r$. We define $\phi_r^i(A) = \ell^i_r$, where $i$ marks the $i$th entry of $\ell_r$, and $\phi_r(A) = \theta_r$. Similarly to how it was done in \cite{beffa2013hamiltonian}, one can see that 
\[
d_e \phi_r^i = E_{1,i+1}, ~~ d_e\phi_r = E_{1,1}.
\]
for any $i=1,\dots, m-1$, and so they are generators for $\h^0$. Also doing similar computations to those in \cite{beffa2013hamiltonian} we readily see that, for $A\in H$
\[
\nabla_r' \phi^i(A) = \begin{pmatrix} 0 & e_i^T \\{\bf 0 } & O\end{pmatrix}, \quad\nabla_r \phi^i(A)= \begin{pmatrix} 0 & e_i^T\Theta^{-1}\\ {\bf 0}& O\end{pmatrix},
\]
where we have used $\nabla_r \phi^i = A_r \nabla_r'\phi^i A^{-1}$ with $\ell_r = 0$ and $\theta_r = 1$. Substituting the values we found in (\ref{twisted}), very clearly 
\[
\{\phi^i, \phi^j\}(A) = 0
\]
for any $A\in H$ since both gradients are strictly upper triangular. These calculations are almost identical to those in \cite{beffa2013hamiltonian}, so we will focus on $\phi_r$, which does not appear in \cite{beffa2013hamiltonian}.

Standard calculations similar to those use to calculate $\nabla_r'\phi^i$ above show that {\it if $A\in H$} then
\[
\nabla'\phi_r(A) = \begin{pmatrix}1 & {\bf 0}^T\\ {\bf 0}&O\end{pmatrix}
\]
for any $r$. For example, we see that
\[
\begin{pmatrix} \theta_r&{\bf 0}^T \\ \ell_r & I_{m-1}\end{pmatrix} \begin{pmatrix} 1&\epsilon v_r\\ {\bf 0}& I_{m-1}\end{pmatrix} = \begin{pmatrix}1&\hq_r^T\\ 0&\widehat{\Theta}_r\end{pmatrix} \begin{pmatrix}\hat\theta_r & 0\\ \hat\ell_r& I_{m-1}\end{pmatrix}
\]
where $\hq_r = \epsilon \theta_r v_r$, $\hat \theta_r = \theta_r - \hq_r^T\hat\ell_r$, $\widehat{\Theta}_r = I_{m-1}+\epsilon\ell_rv_r^T$ and $\hat\ell_r = \widehat{\Theta}_r^{-1}\ell_r$. Differentiating with respect to $\epsilon$, one can readily see that 
\[
\nabla'\phi_r(A) = \begin{pmatrix} \ast&\ast\\ -\theta_r\ell_r&\ast\end{pmatrix}.
\]
Also, the fact that
\[
\begin{pmatrix} \theta_r&{\bf 0}^T \\ \ell_r & I_{m-1}\end{pmatrix}\begin{pmatrix} 1&{\bf 0}^T\\ w_r&I_{m-1}\end{pmatrix} = \begin{pmatrix} \theta_r&{\bf 0}^T \\ \ell_r+ w_r & I_{m-1}\end{pmatrix} 
\]
and 
\[
\begin{pmatrix} \theta_r&{\bf 0}^T \\ \ell_r & I_{m-1}\end{pmatrix}\begin{pmatrix} 1&{\bf 0}^T\\{\bf 0}&\Gamma_r\end{pmatrix} = \begin{pmatrix} 1&{\bf 0}^T\\{\bf 0}&\Gamma_r\end{pmatrix}\begin{pmatrix} \theta_r&{\bf 0}^T \\ \Gamma_r^{-1}\ell_r & I_{m-1}\end{pmatrix} 
\]
imply that the second block column of $\nabla'\phi_r(A)$ vanishes at $\ell_r = 0, \theta_r = 1$, and we get the form above.

Using again that $\nabla \phi_r = A_r \nabla'\phi_r A^{-1}$, $A\in H$, we also have 
\[
\nabla\phi_r(A) = \begin{pmatrix} 1&-q_r^T\\ {\bf 0}&0\end{pmatrix}.
\]
Since all gradients are strictly upper triangular except for $\nabla\phi_r(A)$ which has one nonzero entry in the diagonal, substituting these values in \eqref{twisted}, we get that whenever $A\in H$,
\[
\{\phi, \phi^i\}(A) = 0
\]
for all $i$. This implies that, if we fix $r$, $d_e\{\phi, \phi^i\}(\h) = d_e\{\phi^i, \phi^j\}(\h) = 0$ and so $d_e\{\phi, \phi^i\}, d_e\{\phi^i, \phi^j\}\in \h^0$, which concludes the proof of the theorem.
\end{proof}
 
\subsection{Relationship between the reduced bracket and polygonal evolutions}

Assume $\{\gamma_n\}$ is a non-degenerate twisted $N$-polygon in $\RR^m$ and assume $\GL(m)$ acts on $\RR^m$ via the standard linear action. 
\begin{proposition} The moduli space of non-degenerate twisted polygons under the linear action of $\GL(m)$ can be identified with an open and dense subset of the quotient $\GL^N(m)/H^N$, where $H^N$ acts on $\GL(m)$ via the right discrete gauge action (\ref{gauge}).
\end{proposition}
The proof of this theorem can be found in \cite{MB14} for the general case where $G$ is semisimple and $H$ parabolic. The idea is that given $B\in\GL(m)^N$ we can extend it to a $N$-periodic bi-infinite sequence of elements in $\GL(m)$ and solve the {\it left} discrete system $\T\eta = \eta B$ by recursion, starting at some initial value $\eta_1$. We define the twisted polygon $\gamma = \eta e_1$, $n=1,\dots,N$, twisting it by the monodromy $g = B_N\dots B_2 B_1$ after a period has been completed. Given $\gamma$, we can define its moving frame to be
\[
\rho = (\gamma, \T \gamma, \T^2\gamma, \dots, \T^{m-1}\gamma) \in \GL(m).
\]
Clearly $\eta^{-1}\rho e_1 = \eta^{-1} \gamma = e_1$ and so $\eta_s^{-1}\rho_s\in H$. Furthermore, if $g = \eta^{-1}\rho$, then $A = g^{-1} B \T g$, where $\T \rho = \rho A$ and
\begin{equation}\label{A}
A = \begin{pmatrix} 0&0&\dots &0&a^0\\ 1&0&\dots&0&a^1\\ 0&1&\dots&0&a^2\\\vdots&\vdots&\ddots&\ddots&\vdots\\0&\dots&0&1&a^{m-1}\end{pmatrix}
\end{equation}
defined by the the invariants $a^i$ and the relation $\T^m\gamma = a^{m-1}\T^{m-1}\gamma+\dots+a^1\T\gamma+a^0\gamma$. We obtain the result when we take inverses to revert to right  from left invariant actions.

\begin{corollary}
The reduced bracket defined in the previous section induces a Poisson bracket in the space of difference invariants of twisted polygons in $\RR^m$.
\end{corollary}

\begin{proposition} Let $f, g: U\subset G^N/H^N \to \RR$ be two functions defined locally on the moduli space, and let $F, G: G^N \to \RR$ be extensions of $f$ and $g$, respectively, invariant under the gauge action of $H^N$ on $G^N$. Then, the reduced Poisson bracket defined by \eqref{twisted} can be written as
 \begin{equation}\label{reduceformula}
 \{f, g\}(\a) =  \langle \nabla F - \T \nabla'F, \nabla G\rangle(A) - \frac 12 \langle \nabla F - \T\nabla' F,  \nabla G - \T\nabla' G\rangle(A),
 \end{equation}
 where the relation between $\a$ and $A$ is given by \eqref{A}.
 \end{proposition}
 \begin{proof} First of all, notice that since $F((\T h) L h^{-1})) = F(L)$, we have that $\langle \T(\nabla' F) - \nabla F, \xi\rangle = 0$ for any $\xi\in \h^N$. That is
 \begin{equation}\label{h0}
 \T\nabla' F - \nabla F \in (\h^0)^N,
 \end{equation}
 or, which is the same, $ \nabla'_{s+1} F - \nabla_s F$ is zero except for the first row, for every $s$.
 
 Consider the gradation $\g = \g_+\oplus\g_0\oplus\g_-$, where $\g_-$ is defined by strictly lower triangular matrices, $\g_+$ are strictly upper triangular, and $\g_0$ are diagonal. If $\xi\in \g$, we will split it according to this gradation as $\xi= \xi_++\xi_0+\xi_-$. In this notation, the reduced bracket is given by (\ref{twisted}) applied to two such extensions. Namely
 \begin{eqnarray*}
 \{f,g\}(\a) &=&\frac12\left[\langle (\nabla F)_-, (\nabla G)_+\rangle -\langle (\nabla G)_-, (\nabla F)_+\rangle\right]\\ &+& \frac12\left[\langle (\nabla' F)_-, (\nabla' G)_+\rangle -\langle (\nabla' G)_-, (\nabla' F)_+\rangle\right]\\ &-& \langle \T(\nabla' F)_-, (\nabla G)_+\rangle + \langle \T(\nabla' G)_-, (\nabla F)_+\rangle\\
 &-& \frac 12\sum_{k=1}^m\sum_s\left[( (\nabla_{s+1}' F)_{k,k} (\nabla_s G)_{k,k}) - ((\nabla_{s+1}' G)_{k,k}(\nabla_s F)_{k,k})\right].
 \end{eqnarray*}
From (\ref{h0}) we get that $\T(\nabla'F)_-= (\nabla F)_-$ and $\T(\nabla'F)_{k,k} = (\nabla F)_{k,k}$ for $k=2,\dots,m$. Using this we get
 \begin{eqnarray*}
 \{f,g\}(\a) &=&-\frac12\left[\langle (\nabla F)_-, (\nabla G)_+\rangle -\langle (\nabla G)_-, (\nabla F)_+\rangle\right]\\ &+& \frac12\left[\langle (\nabla' F)_-, (\nabla' G)_+\rangle -\langle (\nabla' G)_-, (\nabla' F)_+\rangle\right]\\  &-& \frac 12\sum_s\left[( (\nabla_{s+1}' F)_{1,1} (\nabla_s G)_{1,1}) - ( (\nabla_{s+1}' G)_{1,1}(\nabla_s F)_{1,1})\right].
 \end{eqnarray*}
We now use the fact that $\langle X, Y\rangle = \langle \T X, \T Y\rangle$ to rewrite \[\langle (\nabla' F)_-, (\nabla' G)_+\rangle = \langle \T(\nabla' F)_-, \T(\nabla' G)_+\rangle = \langle (\nabla F)_-, \T(\nabla' G)_+\rangle.\] This gives
\begin{gather*}
-\frac12\left[\langle (\nabla F)_-, (\nabla G)_+\rangle -\langle (\nabla G)_-, (\nabla F)_+\rangle\right] \\ +\, \frac12\left[\langle (\nabla' F)_-, (\nabla' G)_+\rangle -\langle (\nabla' G)_-, (\nabla' F)_+\rangle\right]
\\
=-\frac12\left[\langle (\nabla F)_-, (\nabla G)_+- \T(\nabla' G)_+\rangle -\langle (\nabla G)_-, (\nabla F)_+- \T(\nabla' F)_+\rangle\right]
\end{gather*}
and from (\ref{h0}) we get that this first portion is equal to
\begin{gather*}
-\frac12\left[\langle \nabla F, \nabla G- \T\nabla' G\rangle -\langle \nabla G, \nabla F- \T\nabla' F\rangle\right]
\\+\,\frac12\sum_s \left[(\nabla_s F)_{1,1}(\nabla_s G- \nabla_{s+1}' G)_{1,1}- (\nabla_s G)_{1,1} (\nabla_s F- \nabla_{s+1}' F)_{1,1}\right].
\end{gather*}
Also, notice that since $\langle\,, \rangle$ is invariant under $G^N$ conjugation, we know that $\langle \nabla F, \nabla G\rangle = \langle \nabla' F, \nabla' G\rangle$. Therefore 
\begin{gather*}
\langle \nabla F, \nabla G- \T\nabla' G\rangle = \langle \T\nabla' F, \T\nabla' G\rangle- \langle \nabla F, \T\nabla' G\rangle = \langle \T\nabla'F-\nabla F, \T\nabla' G\rangle
\\
=\langle \T\nabla'F-\nabla F, \nabla G\rangle +\sum_s (\nabla_{s+1}'F-\nabla_s F)_{1,1}(\nabla_{s+1}'G-\nabla_s G)_{1,1}.
\end{gather*}
Finally, the reduced bracket can be written as
\begin{gather*}
\{f,g\}(\a)  = \langle \nabla F-\T\nabla'F, \nabla G\rangle \\ +\, \frac12\sum_s \left[(\nabla_s F)_{1,1}(\nabla_s G- \nabla_{s+1}' G)_{1,1}- (\nabla_s G)_{1,1} (\nabla_s F- \nabla_{s+1}' F)_{1,1}\right]
\\
-\,\frac 12\sum_s (\nabla_{s+1}'F-\nabla_s F)_{1,1}(\nabla_{s+1}'G-\nabla_s G)_{1,1} \\ -\,\frac 12\sum_s\left[( \nabla_{s+1}' F)_{1,1} (\nabla_s G)_{1,1} -  (\nabla_{s+1}' G)_{1,1}(\nabla_s F)_{1,1}\right]
\\
= \langle \nabla F-\T\nabla'F, \nabla G\rangle -\frac 12\sum_s (\nabla_{s+1}'F-\nabla_s F)_{1,1}(\nabla_{s+1}'G-\nabla_s G)_{1,1},
\end{gather*}
as stated.
 \end{proof}
 Assume next that a twisted polygon $\gamma$ satisfies an evolution of the form
\[
\gamma_t = X
\]
which is invariant under the action of $\GL(m)$.  This is equivalent to saying that  
\[
X = \sum_{k=0}^{m-1} q^k \T^k\gamma
\]
with $q^k$ invariant $N$-vectors under the group action for all $k=0,\dots,m-1$. 
 The next theorem will show that any evolution on the moduli space which is Hamiltonian with respect to the reduced bracket, with Hamiltonian function $f$, is induced by a polygonal evolution associated to a vector field $X^f$ which is algebraically dependent of $\delta f$, and will describe this dependence explicitly.

Let $\gamma$ be a non-degenerate twisted $N$-polygon in $\RR^m$ and let $\rho = (\gamma,\T \gamma,\dots,\T^{n+m-1}\gamma)$ be its left moving frame. Assume 
\begin{equation}\label{evolution}
\gamma_t = X
\end{equation}
is a $\GL(m)$-invariant evolution and let $Q$ be the matrix defined by 
\begin{equation}\label{Qn}
\rho_t = \rho Q.
\end{equation}
That is, 
\[
(\T^k\gamma)_t = \T^kX= \rho Q e_{k+1}.
\]
\begin{theorem}\label{qnabla}
Let $f$ be an invariant function, that is a function of $a^r$ for any $r=0,\dots, m-1$; alternatively a function defined on an open and dense subset of $\GL^N(m)/H^N$. Assume $F$ is an extension of $f$ to $\GL(m)^N$, invariant under the right gauge action of $H^N$. Then, there exists an invariant vector field $X^f$ such that if $\gamma_t = X^f$, the evolution induced on $\GL^N(m)/H^N$ is the reduced Hamiltonian evolution associated to $f$.

The vector field $X^f = \rho Q^f e_1$ is uniquely determined by the condition 
\begin{equation}\label{QF}
\T Q^f({\bf a})e_1 = \frac12 (\nabla F+\T\nabla'F)(A) e_{1}.
\end{equation}
where ${\bf a}$ and $A$ are related as in \eqref{A}. 
\end{theorem}
\begin{lemma}\label{uno}
Let $f$ be a function defined locally on the moduli space $G^N/H^N$ and let $F$ be an extension as in \eqref{h0}. Then
\[
\nabla'F = \begin{pmatrix} -a^0\delta_{a^0} f& -a^0(\delta_{{\bf a}} f)^T\\ \ast&\ast\end{pmatrix}, \quad \nabla F = \begin{pmatrix} \ast & \ast\\ -(\delta_{{\bf a}} f)^T& -a^0\delta_{a^0} f-{\bf a}\cdot \delta_{{\bf a}} f\end{pmatrix}
\]
where $\a = (a^1,\dots,a^{m-1})^T$.
\end{lemma}
\begin{proof}
Assume $F$ is an extension of $f$ to $G^N$ such that $F(A^{-1}) = F(\T h A^{-1} h^{-1}\}) = F(A^{-1}(a))$, where
\[
A^{-1}(a) = \begin{pmatrix} -\a (a^0)^{-1} & I_{m-1}\\(a^0)^{-1}&{\bf 0}^T\end{pmatrix}.
\]
(Notice that we need to invert $A$ in order to consider right gauges.) If $F(A^{-1}) = f(a^0, \a)$, then
\[
F(A^{-1} V) = f(a^0, \a+ v e_i)
\]
whenever 
\[
V = I_m - (a^0)^{-1} v E_{i+1, 1}, \quad i = 1,\dots, m-1.
\]
Differentiating in $v$ we get that 
\[
\delta_{{\bf a}} f v  = -v(a^0)^{-1} \tr(\nabla'F(A^{-1}) E_{i+1,1})
\]
which implies that $\left(\nabla'F(A^{-1})\right)_{1,i+1} = -a^0 \delta_{a^i} f$. Similarly,
\[
F(A^{-1}W) = f(a^0\exp(w))
\]
where 
\[
W = {\rm diag}(\exp(-w), 1, 1, \dots 1).
\]
Differentiating we get 
\[
\delta_{a^0} f a^0w = -w\tr(\nabla'F(A^{-1}) E_{1,1}))
\]
and so $\left(\nabla'F(A^{-1})\right)_{1,1} = -a^0 \delta_{a^0} f$.

We finally use $\nabla F(A^{-1})= A^{-1}\nabla'F(A^{-1})A$ to show that \[e_m^T \nabla F(A^{-1}) e_i = (a^0)^{-1} e_1^T \nabla'F(A^{-1})e_{i+1},\] for $i=1,\dots ,m-1$ and \[e_m^T \nabla F(A^{-1}) e_m =(a^0)^{-1} e_1^T\nabla'F(A^{-1})(\sum_{i=0}^{m-1} a^i e_{i+1}),\] as shown in the statement of the lemma.
\end{proof}

\begin{proof}[Proof  of Theorem \ref{qnabla}] Recall that $\T\rho= \rho A$ and $\rho_t = \rho Q$. This implies that
\begin{equation}\label{KQ}
A^{-1} A_t  = \T Q - A^{-1} Q A
\end{equation}
The LHS of this equality is given by
\begin{equation}\label{KQ}
 A^{-1} A_t = A^{-1} \begin{pmatrix} {\bf 0}^T & a^0_t\\ 0_{m-1} & \a_t\end{pmatrix}=\begin{pmatrix} 0_{m-1} & \a_t- (a^0)^{-1}a^0_t\a\\{\bf 0}^T & (a^0)^{-1}a^0_t\end{pmatrix}.  
 \end{equation}
 Using this expression and the lemma, we get that, if $F$ and $H$ are extension for $f$ and $h$
 \[
\langle \T Q - A^{-1} Q A, \nabla H\rangle = -\sum_{s=1}^N\sum_{i=0}^{m-1} (a_s^i)_t \delta_{a_s^i} h
\]
on the one hand, and on the other hand
\begin{gather*}
\langle \T Q - A^{-1} Q A, \nabla H\rangle  
 = \langle Q,\T^{-1}\nabla H- A \nabla HA^{-1})\rangle = \langle Q, \T^{-1}\nabla H- \nabla' H\rangle.
\end{gather*}
Assume $X^f$ is the vector field inducing the $f$-Hamiltonian evolution on $\a$ and $Q^f$ its associated sequence in $\GL(m)$. Then 
\[
\langle Q^f, \T^{-1}\nabla H- \nabla' H\rangle = \{f, h\}(\a).
\]
On the other hand, from (\ref{reduceformula}) we have that
\[ -\langle \nabla H - \T \nabla'H, \nabla F\rangle(A) + \frac 12 \langle \nabla F - \T\nabla' F,  \nabla H - \T\nabla' H\rangle(A) = \{f,h\}(\a)
\]
for all $h$. From \eqref{h0} we see that this relation holds true under the conditions 
\[
\T Q e_1 = \nabla F e_1 - \frac 12(\nabla F - \T\nabla' F) e_1 = \frac12(\nabla F + \T\nabla' F)e_1
\]
as stated in the theorem.
\end{proof}

\begin{lemma}\label{dos} Assume $\gamma_t = X^f$ induces an $f$-Hamiltonian evolution on the invariants $\a^r$. That is, assume $Q^f$, defined as in \eqref{Qn} satisfies (\ref{QF}). Then
\[
(Q^f)_{1,r} = (\nabla'F)_{1,r} =\T^{-1}(\nabla F)_{1,r} , \quad r=2,\dots,m, \quad\quad (Q^f)_{1,1} = \frac 12\left(\T^{-1}\nabla F+\nabla'F\right)_{1,1}.
\]
\end{lemma}
\begin{proof}
The value of $(Q^f)_{1,1}$ follows directly from (\ref{QF}), while the value for $(Q^f)_{1,r}$ follows directly from \eqref{QF} and \eqref{h0}.\end{proof}

 \section{Connection between both Poisson brackets}

We know return to the definitions and results of Section \ref{sec:scalar}, and in particular the definition of the group of $N$-periodic pseudo-difference operators $\IPSIDO{N}$ as in~\eqref{app:psido}, its algebra $\PSIDO{N}$ and its Poisson structure \eqref{app:ho}. As we showed in Section~\ref{sec:scalar}, the space of $N$-periodic difference operators is a Poisson submanifold of $\PSIDO{N}$, assuming the latter is endowed with Poisson structure \eqref{app:ho}. We will denote that space of difference operators as $\PDO{N}$.
Using the inner product in $\PSIDO{N}$ one can represent the variational derivative of a function $F \colon \IPDO{N} \to \RR$ as an element of $\PSIDO{N}$ in standard fashion. Our next step is to connect them to the context in the previous section. To do this we will consider the subspace $\IDO{N}{m}$ of invertible difference operators of degree $m$.

Let $\F: \IDO{N}{m} \to \RR$ and let $f:\PSL(m)^N/H^N\to \RR$ be defined as
\[
\F(D) = \F\left(\sum^{m}_{r=0} a^r\T^r\right) = f(a^r)
\]
where $a_r$ are the invariants determined by the local section in $\PSL(m)^N/H^N$ and defined by \eqref{A}. From the definition of variational derivative, we can see that
\[
\delta_D F(D) = \sum_{r=0}^m \T^{-r} \delta_{a^r} f(\a)
\]
where $\delta_{a^r} f$ is the standard variational vector derivative of $f$ in the $a^r$ direction. 

We will next slightly restate the definition of \eqref{app:ho} with this notation. If $()_+, ()_-:\PSIDO{N}\to \PSIDO{N}$ are defined as 
\[
 \left(\sum^{m}_{r=-\infty} b^r\T^r\right)_+ =  \sum^{m}_{r=1} b^r\T^r\quad \text{and}\quad  \left(\sum^{m}_{r=-\infty} b^r\T^r\right)_- =  \sum^{-1}_{r=-\infty} b^r\T^r,
 \]
we define the {\it standard scalar $r$-matrix} in $\PSIDO{N}$ as 
\[
r(D) = \frac 12 (D_+-D_-).
\]

Given this standard $r$-matrix, we can define a Poisson bracket on the space $C^\infty(\IPDO{N}, \RR)$. If $\F, \G\in C^\infty(\IPDO{N}, \RR)$ we define
\begin{equation}\label{scalarPoisson}
\{\F, \G\}_1(D) =  \langle r(D\delta_D\F ), D \delta_D\G\rangle - \langle r(\delta_D\F D), \delta_D\G D\rangle,
\end{equation} 
where $\delta_D\F$ and $\delta_D\G$ are evaluated at $D$.
This is the bracket defined in section \ref{sec:scalar}, slightly modified in the notation.
\begin{proposition}\label{scalarreduction} The bracket \eqref{scalarPoisson} can be reduced to the quotient by left action. This quotient can be identified with a section described by the submanifold $b_s^m = -1$ for all $n=1,\dots,N$.
\end{proposition}
\begin{proof}
This is an immediate consequence of the fact that \eqref{scalarPoisson} is invariant under left scalar multiplication, as explained in Proposition \ref{app:tnpoisson}.
\end{proof}
\begin{proposition} Let $\F : \IDO{N}{m} \to \RR$ be a smooth function, invariant under left scalar multiplication. In particular, if $\F(\sum^{m}_{r=0} b^r\T^r) = f(b^r)$ as above, then $f(b^r) = f(b^r/\alpha)$ for any $N$-periodic bi-infinite sequence $\alpha$, $\alpha_s \ne 0$ for all $s$. Then \[\delta_{b^m} f|_{b^m = -1} = \sum_{r=0}^{m-1} b^r \delta_{b^r} f|_{b^m = -1}.\]
\end{proposition}
\begin{proof}
Assume $f$ is as in the statement of the proposition, so that \[f(b^0, b^1, \dots, b^{m-1}, b^m) = f(-b^0/b^m, -b^1/b^m, \dots, -b^{m-1}/b^m, -1).\] Differentiating we get the result of the proposition.\end{proof}
 We will call $\RDO{N}{m}$ the space of difference operators of the form
\begin{equation}\label{D}
D = \sum_{r=0}^{m-1} a^r \T^r - \T^m.
\end{equation}
An immediate corollary is that the reduction of \eqref{scalarPoisson} to $\RDO{N}{m}$ is given by the formula \eqref{scalarPoisson} with 
\begin{equation}\label{scalarred}
\delta_{D} \F = \sum_{r=0}^m \T^{-r} \delta_{a^r} f =  \sum_{r=0}^{m-1} (T^{-m}a^r+\T^{-r}) \delta_{a^r} f.
\end{equation}

\vskip 1ex
 Let $D = \sum_{r=0}^{m-1} a^r \T^r - \T^m$ be the operator associated to twisted projective polygons, such that $D(\gamma) = 0$ for any polygon $\gamma$ with invariants $\{a^r\}$. Define the bi-infinite sequence
\[
d = \det(\gamma, \T \gamma, \dots, \T^{m-1}\gamma).
\]
We say $\gamma$ is {\it non-degenerate} whenever $d_s \ne 0$ for any $s$. Under this identification, we can locally view $\RDO{N}{ m}$ as a section of the moduli space of polygons, under the centro-affine action of $\GL(m)$. 

We now assume that $\gamma_n$ evolves as 
\[
\gamma_t = Y
\]
as described in the previous section. We want to answer the same question as we did before for the moduli space Poisson bracket: can we identify $Y$ so that the evolution induced on $a^r$ is Hamiltonian with respect to \eqref{app:ho}, with Hamiltonian function $\F$? We will denote such a field by $Y^f$.

\begin{proposition} If $\gamma$ is a solution of $\gamma_t = Y^f$, with 
\begin{equation}\label{yf}
Y^f = r(\delta_D\F D)(\gamma)
\end{equation}
 then the invariants $a^r$ satisfy the \eqref{app:ho}-Hamiltonian evolution for $\F$. The vector fields $Y^f$ are unique up to the addition of $g\gamma$, with $g\in \GL(m)$,
\end{proposition}
\begin{proof}
The proof is straightforward. Using \eqref{app:ho} we know that the $\F$-Hamiltonian evolution in $\RDO{N}{m}$ is given by
\[
D_t = r(\delta_D\F D)D- D r(\delta_D\F D).
\]
We also know that $D(\gamma) = 0$, which means 
\[
D_t (\gamma) + D (\gamma_t) = 0.
\]
From here
\[
D_t (\gamma) =  (r(\delta_D\F D)D- D r(\delta_D\F D))(\gamma) = -D r(\delta_D\F D)(\gamma)
\]
and so
\[
D r(\delta_D\F D)(\gamma)
 = -D (\gamma_t).
 \]
Putting both together we get that $\gamma_t = r(\delta_D\F D)(\gamma)$, up to an element in the kernel of $D$, as stated in the proposition.
\end{proof}
\begin{comment}{\rm Notice that the vector field $g\gamma_n$,  $g\ne e$, does not represent an invariant evolution. Indeed, if $\gamma$ is a solution of 
\[
\gamma_t = g\gamma
\]
then $\hat g\gamma$ will not be a solution unless $g$ commutes with $\hat g$, for all $\hat g\in \GL(m)$. Therefore, only adding $\gamma$ will preserve $Y^f$'s invariant property.

On the other hand, if $\gamma_t = \gamma$, then $Q$ as in \eqref{Qn} equals the identity, and from (\ref{KQ}), the evolution induced on the invariants is zero. 
}\end{comment}

To finish this section we will show that $X^f$, as defined in the previous sections, is equal to $Y^f$, for any $f$.

\begin{theorem}\label{XY}
If $X^f$ is defined as in \eqref{QF} and $Y^f$ is defined as in \eqref{yf}, then, 
\[
X^f = Y^f.
\]
\end{theorem}
\begin{proof}
From \eqref{yf} and \eqref{Qn} we know that $Y^f = r(\delta_D\F D)(\gamma)$ and $X^f = \rho Q^f e_1$. We also know that $\{\gamma_s, \gamma_{s+1},\dots, \gamma_{s+m-1}\}$ are linearly independent for any $s$. Therefore, to prove the theorem we will write $Y^f_s$ as a combination of the basis $\{\gamma_s, \gamma_{s+1},\dots, \gamma_{s+m-1}\}$ and we will show that the coefficient of $\gamma_{s+\ell}$ in $r(\delta_{D_s}\F D_s)(\gamma)$ is equal to $e_{\ell+1}^TQ_s^f e_1$. That will conclude the proof.

First we rewrite the operators as
\begin{gather*}
\delta_D\F D = \sum_{r=0}^{m-1}(\T^{-r} + \T^{-m}a^r)\delta_{a^r} f\left(\sum_{k=0}^{m-1}a^k \T^k-\T^m\right) \\ =\, \sum_{r=0}^{m-1}\sum_{k=0}^{m-1}\T^{-m}a^ra^k\delta_{a^r}f\T^k
-\sum_{r=0}^{m-1}\T^{-m}a^r\delta_{a^r}f\T^m+\sum_{r=0}^{m-1}\sum_{k=0}^r \T^{-r}a^k\delta_{a^r}f\T^k \\ +\, \sum_{r=0}^{m-1}\sum_{k=r+1}^{m-1}\T^{-r}a^k\delta_{a^r}f\T^k- \sum_{r=0}^{m-1}\T^{-r}\delta_{a^r}f\T^m.
\end{gather*}
Therefore
\begin{equation}\label{eqr}
\begin{gathered}
2r(\delta_D\F D) 
= -\sum_{r=0}^{m-1}\sum_{k=0}^{m-1}\T^{-m}a^ra^k\delta_{a^r}f\T^k
-\sum_{r=1}^{m-1}\sum_{k=0}^{r-1} \T^{-r}a^k\delta_{a^r}f\T^k \\ +\, \sum_{r=0}^{m-1}\sum_{k=r+1}^{m-1}\T^{-r}a^k\delta_{a^r}f\T^k- \sum_{r=0}^{m-1}\T^{-r}\delta_{a^r}f\T^m.
\end{gathered}
\end{equation}
Using $D(\gamma) = 0$, we can rewrite
\begin{gather*}
- \sum_{r=0}^{m-1}\T^{-r}\delta_{a_s^r}f (\gamma_{s+m})  \\ = - \sum_{r=1}^{m-1}\delta_{a_{s-r}^r}f\gamma_{s+m-r}- \delta_{a_{s}^0}f\sum_{r=0}^{m-1}a_s^r\gamma_{s+r} -\sum_{r=0}^{m-1}\sum_{k=0}^{m-1}\T^{-m}a_s^ra_s^k\delta_{a_s^r}f(\gamma_{s+k})\\  = -\sum_{r=0}^{m-1}a_{s-m}^r\delta_{a_{s-m}^r}f\gamma_s
-\sum_{r=1}^{m-1}\sum_{k=0}^{r-1} \T^{-r}a_s^k\delta_{a^r_s}f(\gamma_{s+k}) \\ = -\sum_{r=1}^{m-1}\delta_{a^r_{s-r}}f\left(\gamma_{s+m-r}-\sum_{k=r}^{m-1} a_{s-r}^s\gamma_{s+k-r}\right).
\end{gather*}
The expression for $2r(\delta_D\F D)(\gamma) $ is now rewritten in terms of our basis. 

We now proceed to compare these to the entries of $Q_s^fe_1$ by matching the coefficient of $\gamma_{s+\ell}$ to the $\ell+1$ entry of $Q_s^fe_1$.
\vskip 1ex
{\it Case $\ell=m-1$:}

One can directly check in the formulas above that the coefficient of $\gamma_{m-1}$ in $2r(\delta_{D_s}F D_s)(\gamma_s)$ is 
\[
-\delta_{a_{s-1}^1}f - a_s^{m-1}\delta_{a_s}^0f - \delta_{a_{s-1}^1}f+a_s^{m-1}\delta_{a_s}^0f = -2\delta_{a_{s-1}^1}f.
\]
On the other hand, from Lemma \ref{uno} and \ref{dos}, we have
\[
e_m^T Q_s^f e_1  = (\nabla_{s-1}\F)_{m,1} = -\delta_{a_{s-1}^1}f
\]
same as that of $r(\delta_{D_s}\F D_s)(\gamma_s)$.
\vskip 1ex
{\it Case $\ell=1,\dots,m-2$:}

Once again we will look at the formula for $2r(\delta_{D_s}\F D_s)(\gamma_s) $ above. We first notice that in order for the two sums in the middle of formula \eqref{eqr} to involve $\gamma_{s+\ell}$, we need $m-r-1>\ell$ since $m-r-1$ is the highest order involved. Straightforward inspection term by term gives us the following coefficients for $\gamma_{s+\ell}$
\begin{gather*}
-\delta_{a_{s-m+\ell}^{m-\ell}} f+\sum_{r=1}^{m-\ell-1}a_{s-r}^{r+\ell}\delta_{a_{s-r}^r}f +\sum_{r=0}^{m-\ell-1}a_{s-r}^{r+\ell}\delta_{a_{s-r}^r}f -\delta_{a_{s-m+\ell}^{m-\ell}} f - a_s^\ell\delta_{a_s^0}f
\\
= -2\delta_{a_{s-m+\ell}^{m-\ell}} f+2\sum_{r=1}^{m-\ell-1}a_{s-r}^{r+\ell}\delta_{a_{s-r}^r}f.
\end{gather*}
If we compute $e_r^T Q_s^f e_1$ for $r=2,\dots,m-1$ using Lemma \ref{dos}, and using that $\nabla_s\F = A_s^{-1} \nabla'_s\F A_s$, we have
\begin{equation}\label{qrfun}
\begin{gathered}
e_r^T Q_s^f e_1 = e_r^T\frac 12(\nabla_{s-1}\F+\nabla_s'\F)e_1 = e_r^T\nabla_{s-1}\F e_1 = e_r^TA_{s-1}^{-1} \nabla'_{s-1}\F A_{s-1}e_1 
\\ = (-a_{s-1}^r(a_{s-1}^0)^{-1}e_1^T + e_{r+1}^T)\nabla'_{s-1}\F e_2 = e_{r+1}^T\nabla_{s-2}\F e_2 +a_{s-1}^r\delta_{a_{s-1}^1}f .
\end{gathered}
\end{equation}
Repeating this process several times we will gradually shift $e_r^T\nabla_{s-1}\F e_1$ down to $e_m^T\nabla_{s-m+r-1}\F e_{m-r+1} = -\delta_{a_{s-m+r-1}^{m-r+1}} f$ (from Lemma \ref{uno}), obtaining
\begin{equation}\label{Qfor}
e_r^T Q_s^f e_1 = \sum_{k=1}^{m-r} a_{s-k}^{r+k-1}\delta_{a_{s-k}^{k}}f -\delta_{a_{s-m+r-1}^{m-r+1}} f.
\end{equation}
We can see that the coefficient of $\gamma_{n+\ell-1}$ is given by choosing $r=\ell$. That is,
\[
\sum_{k=1}^{m-\ell} a_{s-k}^{\ell+k-1}\delta_{a_{s-k}^k}f - \delta_{a_{s-m+\ell-1}^{m-\ell+1}}f
\]
same as that of $r(\delta_{D_s}\F D_s)(\gamma_s)$.
\vskip 1ex
{\it Case $\ell=1$:}

We finally need to calculate the coefficient of $\gamma_s$ in $2r(\delta_{D_s}\F D_s)(\gamma_s)$, and the first entry of $Q_s^fe_1$. Direct inspection as before tells us that the coefficient of $\gamma_s$ in \eqref{eqr} is
\[
\sum_{r=1}^{m-1} a_{s-r}^r \delta_{a_{s-r}^r} f - \sum_{r=0}^{m-1} a_{s-m}^r \delta_{a_{s-m}^r} f - a_s^0\delta_{a_s^0}f.
\]
This time, according to Lemma \ref{uno} and \ref{dos}
\[
e_1^T Q_s^f e_1  = \frac 12e_1^T \nabla_{s-1}\F e_1 - \frac 12 a_s^0\delta_{a_s^0} f.
\]
Using the same reasoning as in \eqref{qrfun}, we obtain that $e_1^T \nabla_{s-1}\F e_1$ is given by a formula similar to \eqref{Qfor} but with $r = 1$. That is
\[
e_1^T \nabla_{s-1}\F e_1 = \sum_{k=1}^{m-1} a_{s-k}^{k}\delta_{a_{s-k}^{k}}f + e_m^T\nabla_{s-m}\F e_m= \sum_{k=1}^{m-1} a_{s-k}^{k}\delta_{a_{s-k}^{k}}f - \sum_{k=0}^{m-1} a_{s-m}^{k}\delta_{a_{s-m}^{k}}f.
\]
Therefore
\[
e_1^T Q_s^f e_1 = \frac12\left(\sum_{k=1}^{m-1} a_{s-k}^{k}\delta_{a_{s-k}^{k}}f - \sum_{k=0}^{m-1} a_{s-m}^{k}\delta_{a_{s-m}^{k}}f - a_s^0\delta_{a_s^0}f\right)
\]
which is the same as that of $r(\delta_{D_s}\F D_s)(\gamma_s)$. This concludes the proof of the theorem.
\end{proof}
\begin{corollary} The reduced bracket obtained in Proposition \ref{scalarreduction}   is Poisson equivalent to the bracket defined by \eqref{reduceformula}. Furthermore, the bracket found in \cite{beffa2013hamiltonian} is Poisson equivalent to the lattice $W_m$-algebra's Poisson bracket, as defined in Proposition \ref{app:tnpoisson}.
\end{corollary}
\begin{proof} The first statement is a straightforward consequence of the previous theorem since $X^f$ induces one of the Poisson brackets on the invariants $a^i$ and $Y^f$ does the some on $D$ as in \eqref{D}.

To show the second part we remark that the bracket in \cite{beffa2013hamiltonian} corresponds to the reduction of \eqref{reduceformula} to the submanifold $a_n^0 = (-1)^{m-1}$, while the lattice $W_m$-algebra bracket is the reduction of the bracket defined in Proposition \ref{scalarreduction} on $D$'s that have the same value for the scalar term. The result follows.
\end{proof}

\bibliographystyle{plain}
\bibliography{latticew.bib}
\end{document}